\documentclass[11pt]{article} % use larger type; default would be 10pt

\usepackage[utf8]{inputenc} % set input encoding (not needed with XeLaTeX)

\usepackage{geometry} % to change the page dimensions
\usepackage{graphicx} % support the \includegraphics command and options
\graphicspath{ {./images/} }

\usepackage{booktabs} % for much better looking tables
\usepackage{array} % for better arrays (eg matrices) in maths
\usepackage{paralist} % very flexible & customisable lists (eg. enumerate/itemize, etc.)
\usepackage{verbatim} % adds environment for commenting out blocks of text & for better verbatim
\usepackage{subfig} % make it possible to include more than one captioned figure/table in a single float
\usepackage{fancyhdr} % This should be set AFTER setting up the page geometry
\usepackage[nottoc,notlof,notlot]{tocbibind} % Put the bibliography in the ToC
\usepackage[titles,subfigure]{tocloft} % Alter the style of the Table of Contents

\usepackage{tikz}
\usetikzlibrary{positioning, fit, shapes}

\usepackage{amssymb}
\usepackage{amsmath}
\usepackage{amsfonts}
\usepackage{mathtools}
\usepackage{planarforest}
\usepackage[ruled,vlined]{algorithm2e}
\usepackage{MnSymbol}
\usepackage{upgreek}
\usepackage[mathscr]{eucal}
\usepackage{algorithmic}
\usepackage{hyperref}
\usepackage{authblk}
\usepackage{multicol}

\newtheorem{theorem}{Theorem}[section]
\newtheorem{lemma}[theorem]{Lemma}
\newtheorem{proposition}[theorem]{Proposition}

\newtheorem{corollary-definition}[theorem]{Corollary-Definition}

\newtheorem{definition}[theorem]{Definition}
\newtheorem{remark}[theorem]{Remark}
\newtheorem{conjecture*}{Conjecture}

\newtheorem{innercustomthm}{Theorem}
\newenvironment{customthm}[1]
  {\renewcommand\theinnercustomthm{#1}\innercustomthm}
  {\endinnercustomthm}

\newcommand{\qed}{\hfill$\square$}

\newcommand{\graft}{\curvearrowright}

\newcommand{\gl}{\diamond}

\newcommand{\dF}{\mathcal{F}}
\newcommand{\Tg}{\mathbf{T}_g}
\newcommand{\Fg}{\mathbf{F}_g}

\newcommand{\EF}{\mathbf{EF}}
\newcommand{\AAb}{\overline{\AA}}

\newcommand{\E}{\ensuremath{\mathbb{E}}\xspace}

\newcommand{\N}{\ensuremath{\mathbb{N}}\xspace}

\renewcommand{\P}{\ensuremath{\mathbb{P}}\xspace}

\newcommand{\R}{\ensuremath{\mathbb{R}}\xspace}

\renewcommand{\AA}{\ensuremath{\mathcal{A}}\xspace}

\newcommand{\LL}{\ensuremath{\mathcal{L}}\xspace}

\newcommand{\NN}{\ensuremath{\mathcal{N}}\xspace}
\newcommand{\OO}{\ensuremath{\mathcal{O}}\xspace}

\renewcommand{\SS}{\ensuremath{\mathcal{S}}\xspace}

\DeclareMathOperator{\Div}{div}

\DeclareMathOperator{\Real}{Re}

\DeclareMathOperator{\ddiv}{div}

\usepackage{tikz-cd}

\newcommand{\mode}[1]{\textcolor{blue}{#1}\index {#1}}

\newcolumntype{C}[1]{>{\centering\let\newline\\\arraybackslash\hspace{0pt}}m{#1}}

\newenvironment{proof}[1][Proof]{\begin{trivlist}
\item[\hskip \labelsep {\bfseries #1}]}{\qed \end{trivlist}}
\newenvironment{example}[1][Example]{\begin{trivlist}
\item[\hskip \labelsep {\bfseries #1}]}{\end{trivlist}}

\title{Efficient Langevin sampling \\with position-dependent diffusion}

\author[1]{Eugen Bronasco}
\author[2]{Benedict Leimkuhler}
\author[3]{Dominic Phillips}
\author[1]{Gilles Vilmart}

\affil[1]{Section of Mathematics, University of Geneva}
\affil[2]{School of Mathematics, University of Edinburgh}
\affil[3]{School of Informatics and Maxwell Institute for the Mathematical Sciences, University of Edinburgh}

\begin{document}
\maketitle

\begin{abstract}
We introduce a numerical method for Brownian dynamics with position dependent diffusion tensor
which is second order accurate for sampling the invariant measure while requiring only one force evaluation per timestep.
Analysis of the sampling bias is performed using the algebraic framework of exotic aromatic Butcher-series.
Numerical experiments confirm the theoretical order of convergence and illustrate the efficiency of the new method.
\end{abstract}


\section{Introduction}
In this paper, we consider the numerical integration of Brownian dynamics 
\begin{equation}\label{eq:SDE1}
    dX(t)= F(X(t))dt+\sigma D(X(t))^{1/2} dW(t),\qquad
    F = -D\nabla V
    +\frac{\sigma^2}2\ddiv(D),
\end{equation}
where $V$ is a smooth potential energy function with Lipschitz continuous gradient,  $W(t)$ is a standard $d$-dimensional Wiener process, $\sigma>0$ is a fixed constant and $D=D(x)$ is a $d\times d$ symmetric and positive definite diffusion matrix that is position dependent.
Such stochastic systems with multiplicative It\^o noise typically arise in 
the context of molecular dynamics \cite{Hummer2005PositiondependentDC, Best_Hummer_2010} and in particular
for overdamped Langevin dynamics, see \cite{Roberts_Stramer_2002,ReyBellet_Spiliopoulos_2016} and recently \cite{Cui_Tong_Zahm_2024,Lelievre_Pavliotis_Robin_Santet_Stoltz_2024,Lelievre_Santet_Stoltz_2024,Lelievre_Santet_Stoltz_2024B}.
Assuming that $V$ grows sufficiently rapidly at infinity and that $D$ is uniformly positive definite, the dynamics defined by  \eqref{eq:SDE1} can be shown to be ergodic and reversible and to admit a unique stationary measure $\mu_{\infty}$ having density $\rho_{\infty}(x) \propto \exp(-\frac2{\sigma^2} V(x))$ \cite[Chap.\thinspace 4]{Pavliotis_2014}.
The use of suitably chosen nonconstant diffusion has been proposed to accelerate sampling of such stationary measures \cite{Roberts_Stramer_2002,ReyBellet_Spiliopoulos_2016}, with considerable attention in the recent literature \cite{Graham_Thiery_Beskos_2022,Cui_Tong_Zahm_2024,Lelievre_Pavliotis_Robin_Santet_Stoltz_2024,Lelievre_Santet_Stoltz_2024,Lelievre_Santet_Stoltz_2024B}.
The challenge addressed in this paper is the efficient generation of samples from the stationary distribution in high dimensions by numerical  discretization of the SDE \eqref{eq:SDE1}.

\subsection{Weakly accurate integration methods for constant diffusion}
In case the diffusion tensor $D$ is the identity matrix, \eqref{eq:SDE1} reduces to the
widely studied overdamped Langevin system
\begin{equation}\label{eq:SDEadditive}
dX(t) = -\nabla V(X(t)) dt + \sigma dW(t).
\end{equation}
%The $d$-dimensional process %$X(t)$ 
% the initial condition $X(0) = X_0$ is assumed deterministic for simplicity, 
%and assumed smooth with $-\nabla V$ globally Lipschitz continuous. 
Setting $\sigma=\sqrt{{2}/\beta}$, we may view $\beta>0$ as a constant proportional to the inverse temperature of a heat bath.
In molecular dynamics models, the dimension $d$ may be very large, as it is proportional to the the number of particles considered in the model. The large dimension combined with possible stiffness issues make the numerical discretization of \eqref{eq:SDEadditive} particularly challenging. 

%Under standard assumptions of smoothness and suitable growth of the potential energy $V$, the system 
% \eqref{eq:SDEadditive} 
%We already gave enough cageats about smoothness and boundedness so I don't think we need to restate them here.
The system (\ref{eq:SDE0}) is \textit{ergodic}, admitting the unique invariant measure $\pi$ given by
\begin{equation}\label{eq:defrho}
\pi(dx) = \rho_\infty(x)dx,\quad \rho_\infty(x) = Z^{-1} e^{2\sigma^{-2}V(x)},\qquad Z=\int_{\mathbb{R}^d} e^{2\sigma^{-2}V(x)} dx,
\end{equation}
and for all initial conditions $X_0$ (assumed deterministic for simplicity) and integrable test functions $\phi$,
\begin{equation}
\label{eq:convtraj}\lim_{T\to\infty} \frac{1}{T} \int_0^T \phi(X(s)) ds = \int_{\R^d} \phi(x) d\pi(x), \quad \text{almost surely}.
\end{equation}
In addition, the expectation of $\phi$ evaluated with the solution $X(t)$ typically converges exponentially fast as $t\rightarrow +\infty$ to the corresponding equilibrium average taken with respect to the invariant measure:
\begin{equation}\label{eq:convexp}
\left |\mathbb{E}\big(\phi(X(t))\big) - \int_{\R^d} \phi(x) d\pi(x) \right | \leq Ce^{-\lambda t}
\end{equation}
for all $t>0$ and constants $C,\lambda>0$ independent of $t$ (but $C$ depends on $\phi$ and $X_0$).

We say that a numerical integrator for \eqref{eq:SDE1} or \eqref{eq:SDEadditive} has weak order $q$ if, for all smooth and integrable test functions $\phi$,
\begin{equation}\label{eq:weakq}
\left | \mathbb{E}\left (\phi(X_n)\right ) - \mathbb{E}\left (\phi(X(t_n))\right ) \right | \leq C(T) h^q,
\end{equation}
for all $t_n=nh\leq T$, all stepsizes $h$ assumed small enough, and any fixed final time $T>0$.
Assuming that it is ergodic, a numerical method has weak order $p$ for the invariant measure
if
\begin{equation}\label{eq:invp}
\left | %\lim_{N\rightarrow +\infty}
\underset{N\rightarrow +\infty}{\lim\,\mathrm{a.s.}}
 \frac{1}{N+1} \sum_{i=0}^N \phi(X_i) - \int_{\R^d} \phi(x) d\pi(x) \right | \leq  C_2 h^p,
\end{equation}
and, analogously to \eqref{eq:convexp}, in many situations one can obtain the estimate
\begin{equation}\label{eq:invpexp}
\left | \mathbb{E}\left (\phi(X_n)\right ) - \int_{\R^d} \phi(x) d\pi(x) \right | \leq C_1e^{-c t_n} + C_2 h^p,
\end{equation}
for all time $t_n=nh$ with constants $c,C_1,C_2>0$ independent from $n$ and $h$.
The simplest method for solving \eqref{eq:SDEadditive} is the Euler-Maruyama method
$$
{X}_{n+1} = {X}_n + h \nabla V(X_n) + \sqrt{h} \sigma  {R_n},
$$
where $R_n \sim \mathcal{N}(0,I_d)$ are independent Gaussian increments.
Under suitable assumptions on the potential $V(x)$, the method has weak order $q=1$ in \eqref{eq:weakq}
%$$
%\left | \mathbb{E}\left (\phi(X_n)\right ) - \mathbb{E}\left (\phi(X(t_n))\right ) | \leq C h,
%$$
and is also order $p=1$ in \eqref{eq:invp} and \eqref{eq:invpexp} with respect to the invariant distribution \cite{Milstein_Tretyakov_2004} and with exponentially fast convergence to equilibrium.
It was shown in \cite{Leimkuhler_Matthews,Leimkuhler_Matthews_Tretyakov} that a second order integrator with respect to the invariant measure (with $p=2$ in \eqref{eq:invp} and  \eqref{eq:invpexp}) can be obtained for \eqref{eq:SDEadditive}  with only one evaluation of the force $-\nabla V$ per timestep,
while the weak order remains $q=1$. 
This method, which we refer to here as the Leimkuhler-Matthews method, is given by the non-Markovian formulation 
\begin{equation}\label{eq:const_D_nonMarc}
    {X}_{n+1} = {X}_n - h \nabla V({X}_n) + \sqrt{h} \sigma  \frac{R_n + R_{n+1}}{2}.
\end{equation}

%$$
%{X}_{n+1} = {X}_n - h \nabla V({X}_n) + %\sqrt{h} \sigma  {R_n},
%$$
Based on the analysis framework of Talay and Tubaro \cite{Talay_Tubaro_1990}, it has been shown that, for large classes or ergodic SDEs, including systems with  degenerate noise and locally Lipschitz fields \cite{Mattingly_Stuart_Higham_2002},  
%for some implicit methods 
and for a broad class of numerical methods, \eqref{eq:weakq} yields \eqref{eq:invp} with $p=q$. 
However, certain schemes can achieve an improved order $p>q$ for sampling the invariant measure. \eqref{eq:const_D_nonMarc} is such a method with $q=1$ and $p=2$;  for any integrable test function $\phi$ and for any time step $h>0$ assumed sufficiently small, 
$$
\left | \mathbb{E}\left (\phi(X_n)\right ) - \mathbb{E}\left (\phi(X(t_n))\right ) \right | \leq C_1e^{-c t_n}  h + C_2 h^2,
$$
for some positive constants $c, C_1, C_2$ independent from $n$ and $h$.
In fact, one can achieve arbitrarily high order $p$ while the weak order remains $q=1$, as discussed in \cite{Abdulle_Vilmart_Zygalakis_15,Bou-Rabee_Owhadi_2010} in the context of splitting methods for underdamped Langevin dynamics and in \cite{Abdulle_Vilmart_Zygalakis_14} for overdamped Langevin dynamics. 
Method \eqref{eq:const_D_nonMarc} can be rewritten in Markovian form using a post-processor \cite{Vilmart_2015}, %$\overline{X}_n = \Psi(X_n)$, 
as
\begin{align}
    X_{n+1} &= X_n -h \nabla V(\overline{X}_n) + \sqrt{h} \sigma R_n, \nonumber \\
    \overline{X}_n &= X_n + \frac{1}{2} \sqrt{h} \sigma R_n. \label{eq:const_scheme}
\end{align}
Indeed, one can show that \eqref{eq:const_D_nonMarc} and \eqref{eq:const_scheme} are equivalent methods.  More precisely, $\overline X_n$ in  \eqref{eq:const_scheme} has the same law as $X_n$ in \eqref{eq:const_D_nonMarc} up to the initial condition choice, which can be shown by substituting $X_n = \overline{X}_n - \frac{1}{2} \sqrt{h} \sigma R_n$ 
and $X_{n+1} = \overline{X}_{n+1} - \frac{1}{2} \sqrt{h} \sigma R_{n+1}$ in the first equation of \eqref{eq:const_scheme}.
We emphasize again that the post-processed value $\overline X_n$ yields order $p=2$ of accuracy for the invariant measure, while the weak order of accuracy for $X_n$ and $\overline X_n$ remains $q=1$ for fixed final time. The reformulation~\eqref{eq:const_scheme} is inspired by the notion of ``effective order'' from the literature of deterministic Runge-Kutta methods \cite{Butcher_1969}.  A similar concept has been found to be useful in the context of stochastic ergodic systems \cite{Vilmart_2015} for the design of high-order samplers based on other schemes, for instance the implicit Euler method, which is applicable to stiff deterministic problems, has been extended to the case of ergodic parabolic SPDEs in \cite{Brehier_Vilmart_2016}.

\subsection{Position-dependent diffusion}
We now consider a symmetric $d\times d$ diffusion matrix $D(x)=(D_{ij}(x))_{i,j=1,\ldots,d}$ with columns $D_j=(D_{ij})_{i=1,\ldots,d}$ for all $j=1,\ldots d$ assumed smooth with respect to $x$.
The {\em divergence} of the smooth matrix $D(x)$ is defined as the vector whose $j$th component is the divergence
$\ddiv D_j(x) = \sum_{i=1}^d \frac{\partial D_{ij}}{\partial x_i}(x)$,
of the $j$th column $D_j$ of the diffusion matrix $D$, viewed as a vector function:
$$
\ddiv (D)  = 
%\big(\sum_{j=1}^d \partial_j D_i^j\big)_{i=1}^d
\begin{pmatrix}
\ddiv D_1 \\ \vdots \\ \ddiv D_d 
\end{pmatrix}.
$$
Defining the symmetric matrix $D(x)$ in the form $D=\Sigma^T\Sigma$, one may also consider the system
\begin{equation}\label{eq:SDE0}
    dX= -(\Sigma^T\Sigma)(X)\nabla V(X)dt+\frac{\sigma^2}2\ddiv(\Sigma^T\Sigma)(X)dt+\sigma \Sigma(X) dW,
\end{equation}
which is equivalent to \eqref{eq:SDE1} with symmetric diffusion matrix $\Sigma=D^{1/2}$. 
Without loss of generality, we shall assume that $\Sigma$ is symmetric.%
%Note however that without loosing generality we need not to assume $\Sigma$ to be symmetric.%
\footnote{Indeed, observe that replacing $\Sigma(x)$ by the symmetric positive definite matrix $(\Sigma(x)^T\Sigma(x))^{1/2}$ in \eqref{eq:SDE0} does not change the law of the solution $X(t)$.} 
Taking as diffusion tensor the identity matrix $D(x)=\Sigma(x)=I_d\in\mathbb{R}^{d\times d}$, the modified systems \eqref{eq:SDE1},\eqref{eq:SDE0} turn out to be equivalent to the additive noise case \eqref{eq:SDEadditive}.

The determination of the diffusion matrix $D(x)$ to achieve certain aims is a challenging question in and of itself which has attracted considerable attention in the recent literature, in particular in the context of high-dimensional problems.
There are many natural choices for $D$,  from both theoretical and practical points of view.  One aim of such a position-dependent diffusion tensor is to tackle the issues of stiff and multi-modal target distributions where regions of high probability are separated by low probability regions  generating metastable trajectories with very low convergence rate $\lambda$ in \eqref{eq:convexp}.
In  recent works \cite{Cui_Tong_Zahm_2024,Lelievre_Pavliotis_Robin_Santet_Stoltz_2024,Lelievre_Santet_Stoltz_2024,Lelievre_Santet_Stoltz_2024B}, optimal diffusion matrices $D$ are proposed with different constraints that maximize the spectral gap of the generator to improve the convergence rate $\lambda$ in \eqref{eq:convexp}. 
As considered in \cite{Girolami_Calderhead_2011, Graham_Thiery_Beskos_2022}, a natural choice for $D(x)$ in the case of a strongly convex potential $V$ is the inverse of the Hessian matrix $\nabla^2 V(x)$ of the potential energy, corresponding to $\Sigma(x)=(\nabla^2 V(x))^{-1/2}$ in \eqref{eq:SDE0}.
%, which can be achieved as $D(x)=L^{-T}(x)$ also using a Cholesky decomposition $\nabla^2 V(x) = L(x)L^T(x)$.
Another natural choice is the diffusion tensor 
\begin{equation}\label{eq:Disotropic}
\Sigma(x)=e^{2\sigma^{-2}V(x)}I_d
\end{equation}
proportional to the identity matrix, corresponding to an isotropic diffusion, which is observed to be close to optimal in the one-dimensional case for various constraints on the diffusion matrix \cite{Lelievre_Santet_Stoltz_2024} and applies efficiently in high-dimensions in \cite{Roberts_Stramer_2002} and \cite[Section 5]{Lelievre_Pavliotis_Robin_Santet_Stoltz_2024}. 
The isotropic form \eqref{eq:Disotropic} is addressed in \cite{Phillips_2024}, where a time transformation of method \eqref{eq:const_D_nonMarc} alternative to the Lamperti transform is introduced and also extends to multivariate diffusion tensors.
The modified dynamics \eqref{eq:SDE0} is also useful in the context of ergodic SPDEs \cite{Debussche_2025} to improve as a preconditioner the regularity of the solution process and hence increase the converge rate of numerical integrators.
Let us also mention that alternative modifications, preserving the invariant measure \eqref{eq:defrho} and different to the form \eqref{eq:SDE1} of Brownian dynamics \eqref{eq:SDE0}, can be considered with oscillatory perturbations that improve the convergence rate to equilibrium, thanks to an enlarged spectral gap in the SDE generator.  Such methods with nonreversible dynamics are proposed in \cite{Lelievre_Nier_Pavliotis_2013} with improved convergence rate
and reduced variance at infinity \cite{Duncan_Lelievre_Pavliotis_2016,Duncan_Nusken_Pavliotis_2017}.
From this perspective, an alternative Stratonovich perturbation is proposed in \cite{Abdulle_Pavliotis_Vilmart_2009} in order to retain the reversibility of the dynamics. However, the stiffness of such highly oscillatory perturbations is tricky to address numerically, as considered recently in \cite{Lu_Spiliopoulos_2018,Zhang_Marzouk_Spiliopoulos_2022}, and such perturbations are not taken up in this paper.

Regardless of the model, there remains the important challenge of accurately discretizing the equations. In this paper, we extend the method \eqref{eq:const_D_nonMarc} to the case of variable diffusion matrices $D(x)$ and introduce a new integrator
for \eqref{eq:SDE1}, equivalently \eqref{eq:SDE0}, that has $O(h^2)$ accuracy with respect to the invariant measure \eqref{eq:deforderp} in arbitrary dimension $d$ while requiring only one evaluation of the modified force  $F=-(\Sigma^T\Sigma)\nabla V+\frac{\sigma^2}2\ddiv(\Sigma^T\Sigma)$ per timestep (and hence one evaluation of the potential gradient $\nabla V$  per timestep).
The following corrected method can be considered for \eqref{eq:SDE1}  as suggested in \cite{Hummer_2018},
\begin{equation}\label{eq:Hummer} 
    {X}_{n+1} = {X}_n + h F({X}_n) +a h \frac{\sigma^2}2 \ddiv(\Sigma^T\Sigma)({X}_n) + \sqrt{h} \sigma  \Sigma(X_n)\frac{R_n + R_{n+1}}{2},
\end{equation}
and $X_n$ which yields order $p=2$ for a constant $\Sigma$ in arbitrary dimension and the value $a=1/4$ yields a consistent method in the one-dimensional case $d=1$, but \eqref{eq:Hummer} is unfortunately not consistent in arbitrary dimension $d$ for a general diffusion tensor. It turns out that deriving such a generalization in arbitrary dimension is not straightforward due to the large number of order conditions that naturally arise in this position-dependent diffusion setting (93 conditions are listed in Remark \ref{rem:93conditions}), as we shall see in the analysis based on the algebraic framework of exotic aromatic trees and Butcher series as introduced in \cite{Laurent_2019} and studied algebraically in \cite{Bronasco22}. 

The main contributions of this paper are as follows:
\begin{itemize}
\item we generalize the method \eqref{eq:const_D_nonMarc} for sampling the invariant measure \eqref{eq:defrho} of Brownian dynamics
\eqref{eq:SDEadditive}
to the case of a position dependent diffusion matrix $D(x)$ in \eqref{eq:SDE1}. In particular, the new method is equivalent to \eqref{eq:const_D_nonMarc} if $D(x)=\Sigma(x)=I_d$ is the identity matrix and equivalent to \eqref{eq:Hummer} if $D(x),\Sigma(x)$ are constant matrices independent of~$x$ (note that in such cases, $\ddiv(D)=\ddiv(\Sigma\Sigma^T)=0$).
\item we prove it has convergence order $p=2$ for sampling the invariant measure \eqref{eq:defrho} in the setting of smooth and globally Lipschitz vector fields, based on the algebraic framework of exotic aromatic Butcher-series, which also gives insight on the method construction.
\item we analyse the mean-square stability properties of the new scheme in the context of stiff problems and discuss possible modifications of the method.
\end{itemize}
This paper is organized as follows. In Section \ref{sec:new_method} we introduce the new second order integrator for sampling the invariant measure in the case of variable diffusion. Section \ref{sec:conv} is dedicated to the convergence analysis while Section \ref{sec:stab} considers  the stability analysis of the new method. Finally, we present numerical experiments in Section \ref{sec:num} that confirm the order of accuracy of the method and show its efficiency particularly with a high dimensional example.

\section{New post-processed method for variable diffusion}
\label{sec:new_method}

We next introduce a generalization of method \eqref{eq:const_scheme} where $\overline X_n$ has order $2$ with respect to the invariant measure for the case of a position-dependent matrix $\Sigma$ in \eqref{eq:SDE0}. The new method is defined as:
\begin{align}
  X_{n+1} &= X_n + hF(\overline{X}_n) + \hat\Phi^\Sigma_h (X_n + \frac{1}{4} h F(\overline{X}_{n-1})), \nonumber \\
    \overline{X}_n &= X_n + \frac{1}{2} \sqrt{h} \sigma \Sigma(X_n) R_n, \quad \text{with } \overline{X}_{-1} = X_0, \label{eq:new_scheme}
\end{align}
where $F=-(\Sigma^T\Sigma)\nabla V+\frac2{\sigma^2}\ddiv(\Sigma^T\Sigma)$ 
and $\Phi^\Sigma_h(X_n) = X_n + \hat\Phi^\Sigma_h (X_n)$ is a one-step integrator of weak order $2$ applied to the SDE problem with and null drift function and the It\^o noise only,
\begin{equation}\label{eq:noiseonly}
dX = \sigma \Sigma(X) dW, 
\end{equation}
where $\Phi^\Sigma (X_0) = X_0 + \sqrt{h} \sigma \Sigma(X_0) R_n + O(h)$ and $R_n \sim \mathcal{N}(0,I_d)$ are independent Gaussian increments. In our work, we refer to this method as the Second-Order Post-processed method for Variable Diffusion (PVD-2). PVD-2 (\ref{eq:new_scheme}) has one evaluation of $F$ per timestep and the number of evaluations of $\Sigma$ depends on the choice of $\Phi^\Sigma_h$.
We remark that it uses a similar post-processor compared to
\eqref{eq:const_scheme},
\begin{equation}\label{eq:postprocessor}
    \overline{X}_n = \Psi_h (X_n) = X_n + \frac{1}{2} \sqrt{h} \sigma \Sigma(X_n) R_n,
\end{equation}
and that it becomes equivalent to \eqref{eq:const_scheme} for the additive noise case $\Sigma(x)=I$.
There are different natural choices for the noise integrator $\Phi_h^\Sigma$ involved in the new scheme \eqref{eq:new_scheme}. Possible weak second order noise integrators for \eqref{eq:noiseonly}
include:

\noindent
\textbf{(Method MT2)}
    The tensor $\Sigma$ is evaluated $5$ times if we choose $\Phi^\Sigma_h$ to be the method from \cite[eq. (3.7)]{Abdulle_Vilmart_Zygalakis_2013}, a derivative-free variant of the so-called Milstein-Talay method \cite[27, p. 103, eq. (2.18)]{Milstein_Tretyakov_2004} written in the particular case of a null drift function:
    \begin{align*}
        X_1 &= X_0 + \frac{1}{2} \sum_{a=1}^d \bigg( \sigma \Sigma_a \big(X_0 + h \sigma \Sigma(X_0)J_a \big) - \sigma \Sigma_a \big(X_0 - h \sigma \Sigma(X_0)J_a \big) \bigg) \\
        &\quad + \frac{\sigma\sqrt{h}}{2} \bigg( \sigma \Sigma \big( X_0 + \sqrt{\frac{h}{2}} \sigma \Sigma(X_0) \chi \big) + \sigma \Sigma \big( X_0 - \sqrt{\frac{h}{2}} \sigma \Sigma(X_0) \chi \big) \bigg) R_n,
    \end{align*}
    where $J_a = (J_{a,b})_{b=1}^d, \chi = (\chi_b)_{b=1}^d,$ and for $a,b = 1, \dots, d$ we have 
    \[ \P(\chi_b = \pm 1) = \frac{1}{2}, \]
    \[ J_{a,b} = \begin{cases} 
        (R_{n,b} R_{n,b} - 1) / 2, & \text{if } a = b, \\ 
        (R_{n,a} R_{n,b} - \chi_a) / 2, & \text{if } a > b, \\ 
        (R_{n,a} R_{n,b} + \chi_b) / 2, & \text{if } a < b. 
    \end{cases}  \]

\noindent
\textbf{(Method W2Ito1)}
    The tensor $\Sigma$ is evaluated $3$ times if we take $\Phi^\Sigma_h$ to be the method introduced in \cite[table 2]{Tang2017} with a null drift function:
    \begin{align*}
        X_1 &= X_0 + \sqrt{h} \sum_{a=1}^d \big( - \sigma \Sigma_a(X_0) + \sigma \Sigma_a(K^{(a)}_1) + \sigma \Sigma_a(K^{(a)}_2) \big) R_{n,a} \\
        &\quad\quad + 2 \sqrt{h} \sum_{a=1}^d \big(\sigma \Sigma_a(X_0) - \sigma \Sigma_a (K^{(a)}_2) \big) \hat{J}_{a,a}, \\
        K^{(a)}_1 &= X_0 + \frac{\sqrt{h}}{2} \sigma \Sigma_a(X_0) \hat{\chi}_1 + \sqrt{h} \sum_{\substack{b=1\\b \neq a}}^d \sigma \Sigma_b(X_0) \hat{J}_{a,b}, \\
        K^{(a)}_2 &= X_0 - \frac{\sqrt{h}}{2} \sigma \Sigma_a(X_0)\hat\chi_1.
    \end{align*} 
    with $\P(\hat{\chi}_i = \pm 1) = \frac{1}{2}$ for $i = 1,2$ and 
    \[ \hat{J}_{a,b} = \begin{cases} 
        \hat{\chi}_1 (R_{n,a}^2 - 1)/2, & \text{if } a = b, \\
        R_{n,b}(1 + \hat{\chi}_2)/2, & \text{if } a > b, \\ 
        R_{n,b}(1 - \hat{\chi}_2)/2, & \text{if } a < b. 
    \end{cases} \]
We emphasize that the above  noise integration methods MT2 and W2Ito1 are  Runge-Kutta type methods with noise increments in the internal stages, an idea first introduced in 
\cite{Roessler_2009} to obtain a number of diffusion tensor evaluations that is independent of the dimension of the noise. We denote versions of the PVD-2 method that use these noise integrators as PVD-2[MT2] and PVD-2[W2Ito1], respectively.

\section{Convergence analysis}
\label{sec:conv}

A numerical method $X_{n+1} = \Phi_h (X_n)$  is of weak order $q$ if for all $t_n=nh\leq T$ where $T>0$ is a fixed final time,
\begin{equation} \label{eq:defweakq}
|\E(\phi(X(t_n)) - \E(\phi(X_n))| \leq C h^q,
\end{equation}
It is said to be ergodic if there exists a unique invariant measure $\pi^h$ satisfying
\[ \lim_{N\to\infty} \frac{1}{N+1} \sum^N_{n=0} \phi(X_n) = \int_{\R^d} \phi(x) d\pi^h(x), \quad \text{almost surely}, \]
and is of order $p$ with respect to the invariant measure if
\begin{equation} \label{eq:deforderp}
\bigg| \int_{\R^d} \phi(x) d\pi^h(x) - \int_{\R^d} \phi(x) d\pi(x) \bigg| \leq Ch^p, 
\end{equation}
where $C$ is independent of $h$ sufficiently small. 

Theorem \ref{thm:main_thm} stands as the main result of the paper.  We show that although the weak order of accuracy of the new method \eqref{eq:new_scheme} applied to \eqref{eq:SDE1} for fixed final time is only $q=1$ in \eqref{eq:defweakq}, the post-processor $\overline X_n$ yields order $p=2$ in \eqref{eq:deforderp}  for sampling the invariant measure. We will introduce the necessary tools and present the proof of Theorem \ref{thm:main_thm} in Section \ref{sec:proof}.

\begin{theorem}\label{thm:main_thm}
    Assume that $V,\Sigma$ are of class $C^\infty$ with all partial derivatives having polynomial growth, and assume that $F,\Sigma$ are globally Lipchitz. 
    Assuming that it is ergodic, PVD-2 given by \eqref{eq:new_scheme} and applied to \eqref{eq:SDE0} has order two with respect to the invariant measure \eqref{eq:defrho}, precisely, 
    \begin{align*}
\bigg| %\lim_{N\to\infty} 
\underset{N\rightarrow +\infty}{\lim\,\mathrm{a.s.}}
\frac{1}{N+1} \sum^N_{n=0} \phi(\overline X_n) - \int_{\R^d} \phi(x) d\pi(x) \bigg| &\leq Ch^2, \\
\big|\E(\phi(\overline X_n)) - \int_{\R^d} \phi(x) d\pi(x) \big| &\leq C(h^2+e^{-c t_n}).
\end{align*}
 for all smooth test function $\phi$ and all initial condition $X_0=x$ where the constant $c, C,\lambda, K$ are independent of $n$, and $h$ assumed small enough.
\end{theorem}

%Section \ref{sec:numerical_experiments} contains numerical experiments. \modg{Details about numerical experiments.}

\subsection{Preliminaries}

Recall that the weak Taylor expansion of the solution of \eqref{eq:SDE0} has the form
\[ \E[\phi(X(h)) | X(0) = X_0] = \phi(X_0) + h(\LL\phi)(X_0) + h^2\frac{\LL^2 \phi}{2!} (X_0) + \cdots + h^k\frac{\LL^k\phi}{k!}(X_0) + \cdots, \]
where the generator of the SDE is given by $\LL \phi := F \cdot \nabla \phi + \frac{\sigma^2}{2} \sum_{a=1}^d \phi^{\prime \prime} (\Sigma_a, \Sigma_a)$. We consider  ergodic numerical methods of the form $X_{n+1} = \Phi_h (X_n)$ with the following weak Taylor expansion,
\[ \E[\phi(\Phi_h(X_0))] = \phi(X_0) + h(\AA_1\phi)(X_0) + h^2(\AA_2\phi)(X_0) + \cdots + h^k(\AA_k \phi)(X_0) + \cdots, \]
where $\AA_k$ are differential operators. 

Let $C^\infty_P(\R^d)$ denote the space of $C^\infty$ functions with all partial derivatives having polynomial growth of the form 
\[ \left|\frac{\partial^n \phi(x)}{\partial x_{i_1} \cdots \partial x_{i_n}}\right| \leq C_n(1 + |x|^{s_n}) \quad \text{for any } n \in \mathbb{N} \text{ and } 1 \leq i_k \leq d \text{ with } k = 1, \dots, n, \]
with constants $C_n$ and $s_n$ independent of $x$. Theorem \ref{thm:ord_cond_inv_meas} presents order conditions with respect to the invariant measure using the differential operators $\AA_k$ assuming $\AA_1 = \LL$.
\begin{theorem}\label{thm:ord_cond_inv_meas}\cite{Abdulle_Vilmart_Zygalakis_14}
    Assume the numerical method $\Phi_h$ has bounded moments of any order along time and admits a weak Taylor expansion with differential operators $\AA_j$. If, for all $\phi \in C_P^\infty(\R^d)$, we have
    \[ \int_{\R^d} (\AA_j\phi)(x) \rho_\infty(x) dx = 0, \quad \text{for } j = 2, \dots, p, \]
    then, the numerical method $\Phi_h$ has order $p$ with respect to the invariant measure.
\end{theorem}
Let us use the following notation for simplicity,
\[ \langle \AA\phi \rangle := \int_{\R^d} (\AA\phi) \rho_\infty dx, \quad \text{for } \phi \in C_P^\infty(\R^d), \]
where $\AA$ is a differential operator. We omit writing $\phi$ when the identity is required to be true for all $\phi \in C^\infty_P(\R^d)$.

We use the following extension of Theorem \ref{thm:ord_cond_inv_meas} combines an integrator with a post-processor to increase the order with respect to the invariant measure.

\begin{theorem}\cite{Vilmart_2015}\label{thm:postprocessor}
    Assume the hypotheses of Theorem \ref{thm:ord_cond_inv_meas} and consider a post-processor $\overline{X}_n = \Psi_h(X_n)$ that admits the following weak Taylor expansion for all $C_P^\infty(\R^d)$:
    \[ \E[\phi(\Psi_h(X_0))] = \phi(X_0) + \sum_{i=1}^{p-1} \alpha_i h^i \LL^i\phi (X_0) + h^p \AAb_p \phi(X_0) + \cdots, \]
    for some constants $\alpha_i$ and differential operator $\AAb_p$. Assume further that
    \[ \langle \AA_{p+1} + [\LL, \AAb_p] \rangle = 0, \]
    where $[\LL, \AAb_p] = \LL \AAb_p - \AAb_p \LL$ is the Lie bracket. Then, $\overline{X}_n$ yields an approximation of order $p+1$ for the invariant measure.
\end{theorem}

\subsection{Integration by parts}\label{sec:IBP}

Since order conditions with respect to the invariant measure are expressed in Theorems \ref{thm:ord_cond_inv_meas} and \ref{thm:postprocessor} using integrals of differential operators applied to test functions $\phi$, we use integration by parts and Lemma \ref{lemma:divD2_D2F} to manipulate the expressions. 

\begin{lemma}\label{lemma:divD2_D2F}
We have the following identities:
\begin{enumerate}
    \item $\Div (\Sigma^2) = \sum_{a=1}^d \Sigma_a \Div(\Sigma_a) + \sum_{a=1}^d \Sigma_a^\prime \Sigma_a, \label{eq:divD2}$
    \item $\Sigma^2 f = \sum_{a=1}^d \Sigma_a (\Sigma_a \cdot f)$.
\end{enumerate}
\end{lemma}

We apply Lemma \ref{lemma:divD2_D2F} in the following example.

\begin{example}
        Let us consider a case with $\Sigma$ assumed constant,
    \begin{align*}
        \sum_{i,j,k,a=1}^d \langle \Sigma_{ia} \Sigma_{ja} &F^k \partial_{i,j,k} \phi \rangle = \sum_{i,j,k,a=1}^d \int_{\R^d} \Sigma_{ia} \Sigma_{ja} F^k \partial_{i,j,k} \phi \rho_\infty dx \tag{A}\\
        \shortintertext{apply integration by parts}
        (A) &= - \sum_{i,j,k,a=1}^d \int_{\R^d} \Sigma_{ia} \Sigma_{ja} (\partial_i F^k) (\partial_{j,k} \phi) \rho_\infty dx \\
        &\quad - \sum_{i,j,k,a=1}^d \int_{\R^d} \Sigma_{ia} \Sigma_{ja} F^k (\partial_{j,k} \phi) (\partial_i \rho_\infty) dx \\
        \shortintertext{we use $\partial_i \rho_\infty = \frac{2}{\sigma^2} f^i \rho_\infty$ and rewrite the expression without the integral notation}
        (A) &= - \sum_{j,k,a=1}^d \langle \sum_{i=1}^d \Sigma_{ia} \Sigma_{ja} (\partial_i F^k) (\partial_{j,k} \phi) + \frac{2}{\sigma^2} (\Sigma_a \cdot f) \Sigma_{ja} F^k (\partial_{j,k} \phi) \rangle
        \shortintertext{applying Lemma \ref{lemma:divD2_D2F} and the definition of $F$ we get}
        (A) &= -\sum_{j,k,a=1}^d \langle \frac{2}{\sigma^2} F^j F^k (\partial_{j,k} \phi) + \sum_{i=1}^d \Sigma_{ia} \Sigma_{ja} (\partial_i F^k) (\partial_{j,k} \phi) \rangle.
    \end{align*}
Next, we consider the same case with $\Sigma$ being non-constant,
\begin{align*}
    \sum_{i,j,k,a=1}^d \langle \Sigma_{ia} \Sigma_{ja} &F^k (\partial_{i,j,k} \phi) \rangle = \sum_{i,j,k,a=1}^d \int_{\R^d} \Sigma_{ia} \Sigma_{ja} F^k (\partial_{i,j,k} \phi) \rho_\infty dx \tag{B} \\
    \shortintertext{apply integration by parts}
    (B) &= - \sum_{i,j,k,a=1}^d \int_{\R^d} (\partial_i \Sigma_{ia}) \Sigma_{ja} F^k (\partial_{j,k} \phi) \rho_\infty dx
     - \sum_{i,j,k,a=1}^d \int_{\R^d} \Sigma_{ia} (\partial_i \Sigma_{ja}) F^k (\partial_{j,k} \phi) \rho_\infty dx \\
    &\quad - \sum_{i,j,k,a=1}^d \int_{\R^d} \Sigma_{ia} \Sigma_{ja} (\partial_i F^k) (\partial_{j,k} \phi) \rho_\infty dx
     - \sum_{i,j,k,a=1}^d \int_{\R^d} \Sigma_{ia} \Sigma_{ja} F^k (\partial_{j,k} \phi) (\partial_i \rho_\infty) dx \\
    \shortintertext{we use $\partial_i \rho_\infty = \frac{2}{\sigma^2} F^i \rho_\infty$ and rewrite the expression without the integral notation}
    (B) &= - \sum_{j,k,a=1}^d \langle (\Sigma_a \Div(\Sigma_a) + \Sigma_a^\prime \Sigma_a)^j F^k (\partial_{j,k} \phi) \\
    &\quad \quad + \sum_{i=1}^d \Sigma_{ia} \Sigma_{ja} (\partial_i F^k) (\partial_{j,k} \phi) + \frac{2}{\sigma^2} (\Sigma_a \cdot F) \Sigma_{ja} F^k (\partial_{j,k} \phi) \rangle
    \shortintertext{applying Lemma \ref{lemma:divD2_D2F} and the definition of $F$ we get}
    (B) &= -\sum_{j,k,a=1}^d \langle \frac{2}{\sigma^2} F^j F^k (\partial_{j,k} \phi) + \sum_{i=1}^d \Sigma_{ia} \Sigma_{ja} (\partial_i F^k) (\partial_{j,k} \phi) \rangle.
\end{align*}
\end{example}

Using the tree formalism introduced in Section \ref{sec:tree_formalism}, we can write this example as 
\[ \langle \dF(\forest{1,1,b}) \rangle = - \langle \dF(2\forest{b,b} + \forest{1,b[1]}) \rangle. \]
An analogous computation can be performed to show that $\langle \dF(\forest{1,b[1]}) \rangle = - \langle \dF(2 \forest{b[b]} + \forest{b[1,1]}) \rangle$.

\subsection{Tree formalism}\label{sec:tree_formalism}

To simplify the computations involving differential operators, we extend the formalisms of grafted  and exotic forests introduced in \cite{Laurent_2019} and studied algebraically in \cite{Bronasco22}.

\begin{definition}
    A grafted tree is a rooted non-planar tree whose vertices are colored by $\bullet$ and leaves are colored by $\bullet$ and $\times$. Grafted forests are monomials of grafted trees.
\end{definition}
The vertices colored by $\times$ are called grafted. The set of grafted trees is denoted by $T_g$ and the corresponding vector space by $\Tg$. Grafted forests are defined as $\Fg := \SS(\Tg)$ with monomials forming a set $F_g$. The size of a grafted forests is computed by taking the sum of the weights of its vertices. Vertices colored by $\bullet$ have weight $1$ and vertices colored by $\times$ have weight $0.5$. All grafted trees up to size $3$ are listed below:
\[ \forest{x}, \quad \forest{b}, \quad \forest{b[x]}, \quad \forest{b[x,x]}, \quad \forest{b[b]}, \quad \forest{b[x,x,x]}, \quad \forest{b[b,x]}, \quad \forest{b[b[x]]}. \]

\begin{definition}
    An exotic forest is a grafted forest in which all grafted vertices are split into pairs.
\end{definition}

The pairing between two grafted vertices is denoted by assigning a natural number to the pair. The choice of the natural number doesn't matter as long as the natural numbers are distinct. An exotic forest is called connected if it cannot be written as a concatenation of non-trivial exotic sub-forests. For example, all connected exotic forests up to size $3$ are:
\[ \forest{b}, \quad \forest{1,1}, \quad \forest{b[b]}, \quad \forest{b[1],1}, \quad \forest{b[1,1]}. \]
The following exotic forests are identical,
\[ \forest{b[1,2,b[2]],1} = \forest{b[4,3,b[4]],3}. \]

We extend the definition of the grafted tree by allowing internal vertices to be colored by $\times$. For example, we allow the following grafted trees up to size $2.5$,
\[ \forest{x}, \quad \forest{b}, \quad \forest{x[x]}, \quad \forest{b[x]}, \quad \forest{x[b]}, \quad \forest{x[x[x]]}, \quad \forest{x[x,x]}, \quad \forest{b[b]}, \quad \forest{b[x[x]]}, \quad \forest{x[b[x]]}, \quad \forest{x[x[b]]}, \quad \forest{b[x,x]}, \quad \forest{x[b,x]}, \quad \forest{x[x[x[x]]]}, \quad \forest{x[x[x,x]]}, \quad \forest{x[x,x[x]]}, \quad \forest{x[x,x,x]}, \]
and the following connected exotic forests up to size $3$,
\[ \forest{b}, \quad \forest{1,1}, \quad \forest{1[1]}, \quad \forest{b[b]}, \quad \forest{b[1,1]}, \quad \forest{b[1],1}, \quad \forest{1[b],1}, \quad \forest{1[2],1,2}, \quad \forest{1[2],1[2]}, \quad \dots. \]

Let $R \sim \NN(0, I)$ denote the Gaussian random variable.

\begin{definition}
    Let $\dF$ be a map that sends grafted forests to the corresponding differential operators. It is defined as
    \[ \dF(\pi)[\cdot] = \sum_{i_r, r \in R(\pi)} \prod_{r \in R(\pi)} \dF(p(r)) [\dF(r)]^{i_r} \partial_{i_r},  \]
    where $R(\pi)$ is the list of roots of $\pi$, $\dF(\bullet) = F(X_0)$, $\dF(\times) = \sigma \Sigma(X_0) R$, and $p(r)$ is the grafted forest connected to $r$ in $\pi$.
\end{definition}
For example, taking $F$ and $\Sigma$ to be evaluated at $X_0$,
\begin{align*}
    \dF(\forest{b[x,x]}) &= \sum_{i=1}^d \dF(\forest{x,x}) [F]^i \partial_i = \sigma^2 \sum_{i,j,k=1}^d (\Sigma R)^j (\Sigma R)^k (\partial_{j,k} F)^i \partial_i, \\
    \dF(\forest{b[x,b],b[x[x]]}) &= \sum_{i,j=1}^d \dF(\forest{x,b})[F]^i \dF(\forest{x[x]})[F]^j \partial_{i,j} \\
    &= \sigma^3 \sum_{i,j,k,l,m,n=1}^d (\Sigma R)^k F^l (\partial_{k,l} F)^i (\Sigma R)^n (\partial_n \Sigma R)^m (\partial_{m} F)^j \partial_{i,j}
\end{align*}

\begin{definition}
    Let $\dF$ be extended to exotic forests $\EF$ by $\dF(k) = \sigma \sum_{a_k=1}^d \Sigma_{a_k}$ for $k \in \N$.
\end{definition}
For example,
\begin{align*}
    \dF(\forest{b[1,1]}) &= \sum_{i=1}^d \dF(\forest{1,1}) [F]^i \partial_i = \sigma^2 \sum_{i,j,k,a=1}^d \Sigma_{ja} \Sigma_{ka} (\partial_{j,k} F)^i \partial_i, \\
    \dF(\forest{b[b[1],1,2],2}) &= \sigma \sum_{i,j,a_2=1}^d \dF(\forest{b[1],1,2})[F]^i \Sigma_{ja_2} \partial_{i,j} \\
    &= \sigma^3 \sum_{i,j,k,l,m,n,a_1,a_2=1}^d \Sigma_{na_1} (\partial_n F)^k \Sigma_{la_1} \Sigma_{ma_2} (\partial_{k,l,m} F)^i \Sigma_{ja_2} \partial_{i,j}.
\end{align*}

The map $\dF$ is an algebraic morphism between the Grossman-Larson algebra of exotic (grafted) forests and the algebra of differential operators with the composition as the product. The Grossman-Larson product is defined as
\[ \pi_1 \gl \pi_2 = \sum_{(\pi_1)} \pi_1^{(1)} (\pi_1^{(2)} \graft \pi_2), \]
where $\Delta(\pi_1) = \sum_{(\pi_1)} \pi_1^{(1)} \otimes \pi_1^{(2)}$ is the deconcatenation coproduct, and $\graft$ is the grafting product that connects all roots of the left operand to vertices of the right operand in all possible ways. For example,
\begin{align*}
    \forest{b,x} \graft \forest{b[x]} &= \forest{b[b,x,x]} + \forest{b[b,x[x]]} + \forest{b[x,x[b]]} + \forest{b[x[b,x]]}, \\
    \forest{b,x} \gl \forest{b[x]} &= \forest{b,x,b[x]} + \forest{b,b[x,x]} + \forest{b,b[x[x]]} + \forest{x,b[b,x]} + \forest{x,b[x[b]]} + \forest{b[b,x,x]} + \forest{b[b,x[x]]} + \forest{b[x,x[b]]} + \forest{b[x[b,x]]}.
\end{align*}

\begin{proposition}\label{prop:dF_gl}\cite{Bronasco22,Laurent_2019}
    Let $\gl$ denote the Grossman-Larson product, then,
    \[ \dF(\pi_1 \gl \pi_2)[\cdot] = \dF(\pi_1)\big[ \dF(\pi_2) [\cdot] \big], \]
    where $\pi_1$ and $\pi_2$ can be either grafted or exotic forests. 
\end{proposition}

We use the tree formalism to simplify the computations, for example, we use exotic forests to express the generator $\LL$ and the weak Taylor expansion of the exact solution as
\[ \LL\phi = \dF(\forest{b} + \frac{1}{2} \forest{1,1})[\phi], \quad \E[\phi(X(h))] = \dF(\exp^\gl(\forest{b} + \frac{1}{2} \forest{1,1}))[\phi]. \]

The integration by parts technique, as described in Section \ref{sec:IBP}, induces a transformation on exotic forests, referred to as IBP \cite{Laurent_2019, Vilmart_2015}.

\begin{theorem}[IBP]\label{thm:IBP}
    Let $\pi \in \EF$ be an exotic forest and choose a grafted root $v$ paired to a grafted leaf $u$. Then,
    \[ \langle \pi \rangle = \langle \sum_{\substack{w \in V(\pi) \\ w \notin \{v, u\}}} \pi^{v \to w} + 2 \pi^\bullet \rangle, \]
    where $V(\pi)$ is the set of vertices of $\pi$, $\pi^{v \to w}$ is the exotic forest $\pi$ in which $v$ is connected to $w$, and $\pi^\bullet$ is the exotic forest $\pi$ in which $v$ is removed and $u$ is replaced by a new black vertex.
\end{theorem}

\subsection{Proof of Theorem \ref{thm:main_thm}}\label{sec:proof}

Let us start by recalling the statement and the setting of Theorem \ref{thm:main_thm}.
We consider the following integrator,
\begin{align}
  X_{n+1} &= X_n + hF(\overline{X}_n) + \hat\Phi^\Sigma_h (X_n + \frac{1}{4} h F(\overline{X}_{n-1})), \nonumber \\
    \overline{X}_n &= X_n + \frac{1}{2} \sqrt{h} \sigma \Sigma(X_n) R_n, \quad \text{with } \overline{X}_{-1} = X_0, \label{eq:new_scheme2}
\end{align}
where $\Phi^\Sigma_h(X_n) = X_n + \hat\Phi^\Sigma_h (X_n) = X_n + \sqrt{h} \sigma \Sigma(X_n) R_n + \OO(h)$ is an integrator of weak order $2$ applied to the problem $dX = \sigma \Sigma(X) dW$.

\begin{customthm}{3.1}
    The integrator of the form (\ref{eq:new_scheme2}) is of order $2$ with respect to the invariant measure.
\end{customthm}

We consider a modification of integrator (\ref{eq:new_scheme2}) which is written as a scheme $X_{n+1} = \Phi_h (X_n)$ and a postprocessor $\overline{X}_n = \Psi_h(X_n)$ with
\begin{align}
    \Phi_h (X_n) &= X_n + hF(Y) + \hat\Phi^\Sigma_h (X_n + \frac{1}{4} h F(X_n)), \nonumber \\
    Y &= X_n + \frac{1}{2} \sqrt{h} \sigma \Sigma(X_n) R_n, \nonumber \\
    \Psi_h (X_n) &= X_n + \frac{1}{2} \sqrt{h} \sigma \Sigma(X_n) R_n, \label{eq:postprocessor2}
\end{align}
where $F(\overline{X}_{n-1})$ of (\ref{eq:new_scheme2}) is replaced by $F(X_n)$. We note that the integrator (\ref{eq:postprocessor2}) requires $2$ evaluations of $F$ per timestep.

\begin{lemma}\label{lemma:A2_cond}
    Given a post-processor of the form (\ref{eq:postprocessor2}) and an integrator $\Phi_h$ which admits weak Taylor expansion with $\AA_1 = \LL$ has order $2$ with respect to the invariant measure if
    \begin{equation}\label{eq:A2_cond}
        \langle \AA_2 \rangle = \langle \dF(\frac{\forest{b,b}}{2} + \frac{\forest{b,1,1}}{2} + \frac{\forest{1[b],1}}{4} + \frac{\forest{1,b[1]}}{2} + \frac{\forest{b[1,1]}}{8} + \frac{\forest{1,1,2,2}}{8} + \frac{\forest{1,2,2[1]}}{2} + \frac{\forest{2[1],2[1]}}{4} + \frac{\forest{2,2[1,1]}}{4}) \rangle.
    \end{equation}
\end{lemma}
\begin{proof}
    By Theorem \ref{thm:postprocessor}, an integrator $\Phi_h$ with $\AA_1 = \LL$ has order $2$ with respect to the invariant measure if $\AA_2$ satisfies
    \begin{equation}\label{eq:A2_general_cond}
        \langle \AA_2 \rangle = \langle \frac{\LL^2}{2} - [\LL, \overline{\AA}_1] \rangle,
    \end{equation}
    where we note that $\langle \frac{\LL^2}{2} \rangle = 0$.
    We use the tree formalism to express the condition (\ref{eq:A2_general_cond}) explicitly. We start by noting that
    \[ \LL = \dF(\forest{b} + \frac{1}{2} \forest{1,1}), \quad \text{and} \quad \AAb_1 = \dF(\frac{1}{8} \forest{1,1}). \]
    This implies, using Proposition \ref{prop:dF_gl}, that 
    \[ [\LL, \AAb_1] = \dF(\frac{1}{4} \forest{1[b],1} - \frac{1}{4} \forest{1,b[1]} - \frac{1}{8} \forest{b[1,1]}). \]
    Using Proposition \ref{prop:dF_gl}, we express $\LL^2$ as
    \[ \LL^2 = \dF(\forest{b,b} + \forest{b[b]} + \forest{b,1,1} + \forest{1[b],1} + \forest{1,b[1]} + \frac{1}{2} \forest{b[1,1]} + \frac{1}{4} \forest{1,1,2,2} + \forest{1,2,2[1]} + \frac{1}{2} \forest{2[1],2[1]} + \frac{1}{2} \forest{2,2[1,1]}). \]
    This gives us the following condition on $\AA_2$:
    \[ \langle \AA_2 \rangle = \langle \dF( \frac{\forest{b,b}}{2} + \frac{\forest{b[b]}}{2} + \frac{\forest{b,1,1}}{2} + \frac{\forest{1[b],1}}{4} + \frac{3 \forest{1,b[1]}}{4} + \frac{3 \forest{b[1,1]}}{8} + \frac{\forest{1,1,2,2}}{8} + \frac{\forest{1,2,2[1]}}{2} + \frac{\forest{2[1],2[1]}}{4} + \frac{\forest{2,2[1,1]}}{4} ) \rangle. \]
    We use IBP from Theorem \ref{thm:IBP} with $\langle \dF(\frac{\forest{1,b[1]}}{4}) \rangle = - \langle \dF( \frac{\forest{b[b]}}{2} + \frac{\forest{b[1,1]}}{4}) \rangle$ as follows
    \[ \langle \AA_2 \rangle = \langle \dF( \frac{\forest{b,b}}{2} + \frac{\forest{b,1,1}}{2} + \frac{\forest{1[b],1}}{4} + \frac{\forest{1,b[1]}}{2} + \frac{\forest{b[1,1]}}{8} + \frac{\forest{1,1,2,2}}{8} + \frac{\forest{1,2,2[1]}}{2} + \frac{\forest{2[1],2[1]}}{4} + \frac{\forest{2,2[1,1]}}{4} ) \rangle. \]
    This finishes the proof.
\end{proof}

\begin{proposition}
  The post-processed integrator $\overline{X}_n = (\Psi_h \circ \Phi_h^n) (X_0)$ of (\ref{eq:postprocessor2}) is of order $2$ with respect to the invariant measure.
\end{proposition}
\begin{proof}
  The post-processed integrator $\overline{X}_n = (\Psi_h \circ \Phi_h^n) (X_0)$ of (\ref{eq:postprocessor2}) has $\AA_1 = \LL$. To see that it is of order $2$ with respect to the invariant measure, we use Lemma \ref{lemma:A2_cond}, that is, we check that $\AA_2$ satisfies (\ref{eq:A2_cond}). The differential operators appearing in the weak Taylor expansion of $\Phi^\Sigma_h$ are denoted by $\AA^\Sigma_j$ where $j\in\N$ with $\AA^\Sigma_1 = \dF(\frac{1}{2}\forest{1,1})$. Therefore, the differential operator $\AA_2$ has the form
  \[ \AA_2 = \dF(\frac{1}{2} \forest{b,b} + \frac{1}{2} \forest{b,1,1} + \frac{1}{2} \forest{1,b[1]} + \frac{1}{8} \forest{b[1,1]}) + \AA^\Sigma_2. \]
  Since $\Phi^\Sigma_h$ is weak order $2$, $\AA_2^\Sigma$ of $\Phi^\Sigma_h (X_n + \frac{1}{4} h F(X_n))$ is
  \[ \AA_2^\Sigma = \dF(\frac{\forest{1[b],1}}{4} + \frac{\forest{1,1,2,2}}{8} + \frac{\forest{1,2,2[1]}}{2} + \frac{\forest{2[1],2[1]}}{4} + \frac{\forest{2,2[1,1]}}{4}). \]
  Therefore, the condition (\ref{eq:A2_cond}) is satisfied and the method has order $2$ with respect to the invariant measure.
\end{proof}

We are now ready to prove Theorem \ref{thm:main_thm} and to show that the post-processed integrator of (\ref{eq:new_scheme2}) has order $2$ with respect to the invariant measure.

\begin{proof}\textbf{of Theorem \ref{thm:main_thm}}
  We recall that the only difference between the integrators described by (\ref{eq:new_scheme2}) and (\ref{eq:postprocessor2}) is the replacement of $F(\overline{X}_{n-1})$ by $F(X_n)$ which simplifies the analysis. Let us now show that the differential operators $\AA_2$ of (\ref{eq:new_scheme2}) and (\ref{eq:postprocessor2}) are identical and, therefore, the integrator (\ref{eq:new_scheme2}) is of order $2$ with respect to the invariant measure.

  We have the following identity
  \begin{align*}
    \overline{X}_{n-1} &= X_{n-1} + \frac{1}{2} \sqrt{h} \sigma \Sigma(X_{n-1}) R_{n-1} \\
    &= X_n - h F(\overline{X}_{n-1}) - \hat\Phi^\Sigma_h(X_{n-1} + \frac{1}{4} h F(\overline{X}_{n-2})) + \frac{1}{2} \sqrt{h} \sigma \Sigma(X_{n-1}) R_{n-1} \\
    &= X_n - \frac{1}{2} \sqrt{h} \sigma \Sigma(X_n) R_{n-1} + \OO(h).
  \end{align*}
  Therefore, the $\hat\Phi^\Sigma_h$ term of (\ref{eq:new_scheme2}) has the following form
  \[ \hat\Phi^\Sigma_h \big( X_n + \frac{1}{4} h F(\overline{X}_{n-1}) \big) = \hat\Phi^\Sigma_h \big( X_n + \frac{1}{4} h F(X_n) - \frac{1}{8} h \sqrt{h} \sigma F^\prime (X_n) \Sigma(X_n) R_{n-1} + \OO(h^2) \big). \]
  We see that the Taylor expansion of the $\hat\Phi^\Sigma_h$ term of (\ref{eq:new_scheme2}) differs from the Taylor expansion of the $\hat\Phi^\Sigma_h$ term of (\ref{eq:postprocessor2}) by
  \[ - \frac{1}{8} h^2 \sigma^2 (\Sigma R_n)^\prime F^\prime \Sigma R_{n-1} + \OO(h^{2.5}), \]
  which has expectation $\OO(h^3)$, and, thus, the differential operators $\AA_2$ of (\ref{eq:new_scheme2}) and (\ref{eq:postprocessor2}) are identical.
\end{proof}

\begin{remark} \label{rem:93conditions}
    A direct approach, involving the computation of order conditions for second order with respect to the invariant measure, followed by the solution of the resulting system, proved to be too challenging to perform manually as it required solving a system of 93 order conditions. A subset of these conditions is listed below:
    \begin{multicols}{2}
    \begin{enumerate}
        \item $a_\sigma(\forest{b,1,1}) - 2 a_\sigma (\forest{1,1,2,2}) = 0$,
        \item $a_\sigma(\forest{1,2[1],2}) - 2 a_\sigma (\forest{1,1,2,2}) = 0$,
        \item $a_\sigma(\forest{1[1],2,2}) - 2 a_\sigma (\forest{1,1,2,2}) = 0$,
        \item $a_\sigma(\forest{1[b],1}) = 0$,
        \item $a_\sigma(\forest{b[1],1}) = 0$,
        \item[$\cdots$]
        \item[91.] $a_\sigma(\forest{b[1,1]}) = 0$,
        \item[92.] $a_\sigma(\forest{1[b[1]]}) = 0$,
        \item[93.] $a_\sigma(\forest{1,(1,b)}) = 0$,
    \end{enumerate}
    \end{multicols}
\noindent    where $a_\sigma(\pi)$ denotes the Runge-Kutta coefficient corresponding to the forest $\pi$ and divided by its symmetry.
\end{remark}

\section{Mean-square stability analysis}
\label{sec:stab}

We observe that the next step, \( X_{n+1} \), in the PVD-2 method (\ref{eq:new_scheme}) proposed here is dependent on both \( X_n \) and \( \overline{X}_{n-1} \). To analyze the stability of this method effectively, we express it in a partitioned form \( X^P_{n+1} = \Phi^P_h (X^P_n) \) where \( X^P_n = (X_n^T, \overline{X}^T_{n-1})^T \):

\begin{equation}\label{eq:partitioned}
\begin{pmatrix}
    X_{n+1} \\ \overline{X_n}
\end{pmatrix}
=
\Phi^P_h \begin{pmatrix}
    X_n \\ \overline{X}_{n-1}
\end{pmatrix}
=
\begin{pmatrix}
    X_n + hF(\overline{X}_n) + \hat\Phi^\Sigma_h \left(X_n + \frac{1}{4} hF(\overline{X}_{n-1})\right) \\
    X_n + \frac{1}{2} \sqrt{h} \sigma \Sigma(X_n) R_n
\end{pmatrix}.
\end{equation}

\subsection{Stability domain for mean-square stiff problems}
We consider the following test problem in dimension \( d = 1 \), which is introduced in \cite{Saito_Mitsui} and widely used in the literature \cite{SKROCK,Burrage_Burrage,Higham,Tocino} for studying the mean-square stability of integrators applied to stiff problems: \begin{equation}\label{eq:testproblem}
 dX(t) = \lambda X(t) dt + \mu X(t) dW(t), \quad X(0) = 1, 
\end{equation}
where \( \lambda \) and \( \mu \) are fixed complex parameters. After applying the new method to the test problem, we obtain the stability matrix \( R(p,q,R_n) \) of the following form, with \( p = \lambda h \) and \( q = \mu \sqrt{h} \),

\begin{equation}\label{eq:stability_matrix}
\begin{pmatrix}
    X_{n+1} \\ \overline{X}_n
\end{pmatrix}
=
\begin{pmatrix}
    1 + p + \frac{1}{2} p q R_n + \hat{R}^\Sigma (p,q,R_n) & \frac{1}{4} p \hat{R}^\Sigma(p,q,R_n) \\
    1 + \frac{1}{2} q R_n & 0
\end{pmatrix}
\begin{pmatrix}
    X_n \\ \overline{X}_{n-1}
\end{pmatrix},
\end{equation}

where \( R^\Sigma(p,q,R_n) \) is the stability function of the noise integrator and \( \hat{R}^\Sigma = R^\Sigma - 1 \).
Following the ideas from Saito-Mitsui \cite{Saito_Mitsui}, we consider \( \E[X_{n+1}^P {X_{n+1}^P}^T] \) and obtain the following equation with $R = R(p, q, R_n)$ being the stability matrix from (\ref{eq:stability_matrix}),

\begin{equation}\label{eq:E_stability_matrix}
\begin{pmatrix}
    \E[X_{n+1}^2] \\ \E[\overline{X}^2_n] \\ \E[X_{n+1} \overline{X}_n]
\end{pmatrix}
=
\begin{pmatrix}
    \E[R_{11}^2] & \E[R_{12}^2] & 2 \E[R_{11} R_{12}] \\
    \E[R_{21}^2] & 0            & 0                   \\
    \E[R_{21} R_{11}] & 0       & \E[R_{21} R_{12}]
\end{pmatrix}
\begin{pmatrix}
    \E[X_n^2] \\ \E[\overline{X}^2_{n-1}] \\ \E[X_n \overline{X}_{n-1}]
\end{pmatrix}.
\end{equation}

The mean-square stability region of method (\ref{eq:new_scheme}), that is the domain of $p,q$ such that the second moments of the numerical solution $\E(X_n^2),\E(\overline{X}_n^2)$ tend to 0 as $n\rightarrow +\infty$, is then computed by checking the values of \( p \) and \( q \) for which the largest eigenvalue of the matrix in (\ref{eq:E_stability_matrix}) is smaller than \( 1 \). To do this, we need to choose the noise integrator. We note that both noise integrators from Section \ref{sec:new_method} have the following stability function
\begin{equation}\label{eq:noise_stability_function}
    R^\Sigma(p,q,R_n) = 1 + q R_n + \frac{q^2}{2} (R_n^2 - 1).
\end{equation}
The resulting stability region is computed numerically for real $p$ and is presented in Figure \ref{fig:stability_region_1}. For comparison, the light gray region in Figure \ref{fig:stability_region_1} is the stability region of the exact solution for which $\E(X(t)^2)\rightarrow 0$ as $t\rightarrow +\infty$ and is given by the condition $\Real(\lambda)+\frac{|\mu|^2}2< 0$.

\begin{figure}
    \centering
    \subfloat[New method (\ref{eq:new_scheme})]{
        \includegraphics[width=0.3\textwidth]{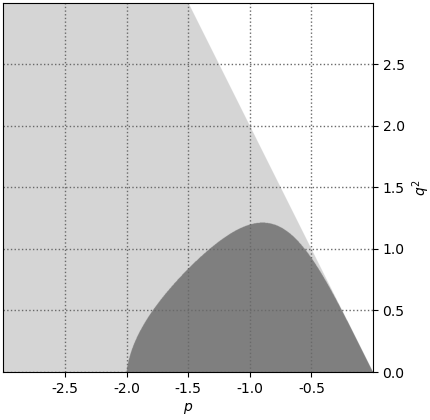}
        \label{fig:stability_region_1}
    }
    \subfloat[Modification 1 (\ref{eq:modification_1})]{
        \includegraphics[width=0.3\textwidth]{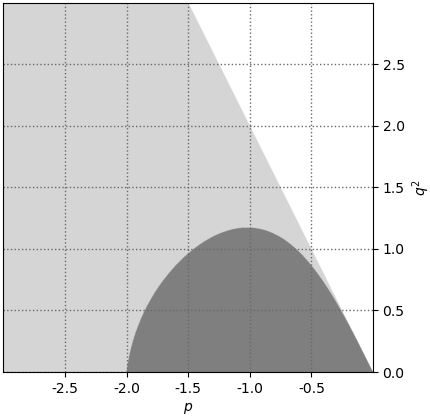}
        \label{fig:stability_region_2}
    }
    \subfloat[Modification 2 (\ref{eq:modification_2})]{
        \includegraphics[width=0.3\textwidth]{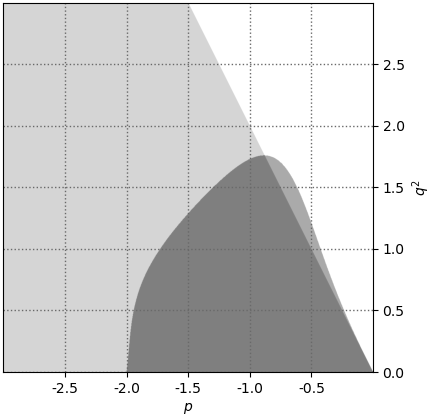}
        \label{fig:stability_region_3}
    }
    \caption{Mean-square stability domains of the PVD-2 method (\ref{eq:new_scheme}) and modifications (\ref{eq:modification_1}) and (\ref{eq:modification_2}), 		
		for which $\E(X_n^2) \rightarrow 0$ for the scalar test problem  \eqref{eq:testproblem} in the $(p,q^2)$--plane where $p=\lambda h,q=\mu \sqrt h$.
		}
\end{figure}

\subsection{Improving the stability}

We consider slight modifications to the method to enhance the stability region without compromising the order of convergence or significantly increasing the computational cost. The first modification brings the term \(hF(\overline{X}_{n-1})\) inside the \(\sqrt{h}\) term of the Taylor expansion of the noise integrator. This is enough to obtain the desired order $2$ and simplifies the stability matrix in (\ref{eq:stability_matrix}). We obtain the following method
\begin{align}
  X_{n+1} &= X_n + hF(\overline{X}_n) + \hat\Phi^\Sigma_h (X_n, X_n + \frac{h}{4} F(\overline{X}_{n-1})), \nonumber \\
    \overline{X}_n &= X_n + \frac{1}{2} \sqrt{h} \sigma \Sigma(X_n) R_n, \quad \text{with } \overline{X}_{-1} = X_0, \label{eq:modification_1}
\end{align}
with the following choice of the noise integrator \(\Phi_h^\Sigma (X_n, X_n^{(1)})\):
\begin{enumerate}
    \item Modified weak order \( 2 \) method from \cite{Abdulle_Vilmart_Zygalakis_2013}
    \begin{align*}
        \Phi_h^\Sigma (X_n, X_n^{(1)}) &= X_n + \frac{1}{2} \sum_{a=1}^d \bigg( \sigma \Sigma_a \big(X_n + \sigma \Sigma(X_n)J_a \big) - \sigma \Sigma_a \big(X_n - \sigma \Sigma(X_n)J_a \big) \bigg) \\
        &\quad + \frac{\sigma\sqrt{h}}{2} \bigg( \sigma \Sigma \big( X_n^{(1)} + \sqrt{\frac{h}{2}} \sigma \Sigma(X_n) \chi \big) + \sigma \Sigma \big( X_n^{(1)} - \sqrt{\frac{h}{2}} \sigma \Sigma(X_n) \chi \big) \bigg) R_n.
    \end{align*}
    
    \item Modified W2Ito1 method from \cite{Tang2017}
    \begin{align*}
        \Phi_h^\Sigma (X_n, X_n^{(1)}) &= X_n + \sqrt{h} \sum_{a=1}^d \big( - \sigma \Sigma_a(X_n) + \sigma \Sigma_a(K^{(a)}_1) + \sigma \Sigma_a(K^{(a)}_2) \big) R_{n,a} \\
        &\quad\quad + 2 \sqrt{h} \sum_{a=1}^d \big(\sigma \Sigma_a(X_n) - \sigma \Sigma_a (K^{(a)}_2) \big) \hat{J}_{a,a}, \\
        K^{(a)}_1 &= X^{(1)}_n + \frac{\sqrt{h}}{2} \sigma \Sigma_a(X_n) \hat{\chi}_1 + \sqrt{h} \sum_{\substack{b=1\\b \neq a}}^d \sigma \Sigma_b(X_n) \hat{J}_{a,b}, \\
        K^{(a)}_2 &= X_n - \frac{\sqrt{h}}{2} \sigma \Sigma_a(X_n)\hat\chi_1.
    \end{align*}
\end{enumerate}
The notation coincides with that used in the examples of Section \ref{sec:new_method}. The improved stability region can be found in Figure \ref{fig:stability_region_2}.
The next improvement of the stability region is achieved by modifying the Milstein-Tretyakov term of the noise integrator. We note that the term \(\frac{q^2}{2} (R_n^2 - 1)\) from (\ref{eq:noise_stability_function}) results in the term \(q^4\) in \( \E[R_{11}^2], \E[R_{12}^2], \E[R_{11} R_{12}] \) from (\ref{eq:E_stability_matrix}). We decrease the significance of this term by multiplying it by \(1 + \frac{p}{2}\) which goes to \(0\) as \(p\) approaches \(-2\). This is achieved by replacing \( X_n \) corresponding to the Milstein-Tretyakov term by \( X_n + \frac{h}{2} F(\overline{X}_n) \). The updated method is
\begin{align}
  X_{n+1} &= X_n + hF(\overline{X}_n) + \hat\Phi^\Sigma_h (X_n, X_n + \frac{h}{4} F(\overline{X}_{n-1}), X_n + \frac{h}{2} F(\overline{X}_n)), \nonumber \\
    \overline{X}_n &= X_n + \frac{1}{2} \sqrt{h} \sigma \Sigma(X_n) R_n, \quad \text{with } \overline{X}_{-1} = X_0, \label{eq:modification_2}
\end{align}
with the noise integrator being one of the following options:
\begin{enumerate}
    \item Modified weak order \( 2 \) method from \cite{Abdulle_Vilmart_Zygalakis_2013}
    \begin{align*}
        \Phi_h^\Sigma (X_n, X_n^{(1)}, X_n^{(2)}) &= X_n + \frac{1}{2} \sum_{a=1}^d \bigg( \sigma \Sigma_a \big(X_n + \sigma \Sigma(X_n^{(2)})J_a \big) - \sigma \Sigma_a \big(X_n - \sigma \Sigma(X_n^{(2)})J_a \big) \bigg) \\
        &\quad + \frac{\sigma\sqrt{h}}{2} \bigg( \sigma \Sigma \big( X_n^{(1)} + \sqrt{\frac{h}{2}} \sigma \Sigma(X_n^{(2)}) \chi \big) + \sigma \Sigma \big( X_n^{(1)} - \sqrt{\frac{h}{2}} \sigma \Sigma(X_n^{(2)}) \chi \big) \bigg) R_n.
    \end{align*}
    
    \item Modified W2Ito1 method from \cite{Tang2017}
    \begin{align*}
        \Phi_h^\Sigma (X_n, X_n^{(1)}, X_n^{(2)}) &= X_n + \sqrt{h} \sum_{a=1}^d \big( - \sigma \Sigma_a(X_n) + \sigma \Sigma_a(K^{(a)}_1) + \sigma \Sigma_a(K^{(a)}_2) \big) R_{n,a} \\
        &\quad\quad + 2 \sqrt{h} \sum_{a=1}^d \big(\sigma \Sigma_a(X_n) - \sigma \Sigma_a (K^{(a)}_2) \big) \hat{J}_{a,a}, \\
        K^{(a)}_1 &= X_n^{(1)} + \frac{\sqrt{h}}{2} \sigma \Sigma_a(X_n^{(2)}) \hat{\chi}_1 + \sqrt{h} \sum_{\substack{b=1\\b \neq a}}^d \sigma \Sigma_b(X_n^{(2)}) \hat{J}_{a,b}, \\
        K^{(a)}_2 &= X_n - \frac{\sqrt{h}}{2} \sigma \Sigma_a(X_n^{(2)})\hat\chi_1.
    \end{align*}
\end{enumerate}

Further experiments that modified the coefficient \(\frac{1}{2}\) of \( F(\overline{X}_n) \) confirmed that \( X_n + \frac{h}{2} F(\overline{X}_n) \) is the optimal choice for \( X_n^{(2)} \). The updated stability region is shown in Figure~\ref{fig:stability_region_3}. However, in our broader experiments, modifications \eqref{eq:modification_1} and \eqref{eq:modification_2} did not lead to a significant improvement in stability compared to the original version \eqref{eq:new_scheme}. As a result, we opted to use for simplicity the original version \eqref{eq:new_scheme} in our numerical simulations.

\section{Numerical experiments}
\label{sec:num}

We present experiments that confirm the convergence  order of two of PVD-2 for sampling the invariant measure. %We show the convergence by comparing both the mean $L_1$ error of the invariant measure in the infinite-time limit and by computing the average of the observable $||\cdot||_2^2$ at a finite observation time. 
We explore several one and two-dimensional problems as well as higher dimensional problems to emphasize that the method converges regardless of dimensionality. In the following, we fix  $\sigma=1$. 

In our experiments, we compare the performance of PVD-2, given by \eqref{eq:new_scheme}, against the following methods: Euler-Maruyama (EM), Leimkuhler-Matthews with drift correction (LMd) (referred to as Hummer-Leimkuhler-Matthews in \cite{Phillips_2024}), the Strang splitting between Runge-Kutta 4 (explicit of order 4) and W2Ito1 noise integrator (RK4[W2Ito1]) \cite{Tang2017}, and Leimkuhler-Matthews with time rescaling (LMt) \cite{Phillips_2024}. Properties of these methods are summarized in Table \ref{tab:method_summary}, below.

\begin{table}[h!]
    \centering
    \renewcommand{\arraystretch}{1.2}
    \begin{tabular}{|l|c|c|c|c|}
        \hline
        \textbf{Method} & \textbf{Weak Order} & \textbf{Sampling Order} & \textbf{\# $F$ Eval.} & \textbf{\# $\Sigma$ Eval.}  \\
        \hline
        EM & 1 & 1 & 1 & 1  \\
        \hline
        LMd (dim. $d=1$) & 1 & 1 & 1 & 1  \\
        \hline
        RK4[W2Ito1] & 2 & 2 & 8 & 3  \\
        \hline
        LMt & 1 & 2 & 1 & 1 \\
        \hline
        PVD-2[W2Ito1] & 1 (expected) & 2 (expected) & 1 & 3 \\
        \hline
        PVD-2[MT2] & 1 (expected) & 2 (expected) & 1 & 5  \\
        \hline
    \end{tabular}
    \caption{Summary of method characteristics, including the number of $F$ and $\Sigma$ evaluations per step. \textit{Sampling order} is the order of sampling of the invariant measure, i.e. $p$ in equation \eqref{eq:invpexp}. }
    \label{tab:method_summary}
\end{table}

Note that LMd does not converge for general variable diffusion in dimensions larger than $d=1$. The method RK4[W2Ito1] uses Strang splitting, doubling the number of force evaluations, giving a total of $4 \times 2 = 8$. The same order could be achieved with fewer $F$ evaluations, however, our aim is to compare against a highly-accurate integrator as a challenging baseline. Method LMt uses a constant stepsize $h$ in a transformed time variable $\tau(t)$. For plots, we display error curves with an effective step size $h' = h \left\langle \frac{dt}{d\tau} \right\rangle$, where $\left<\cdot\right>$ denotes a trajectory average. The number of $\Sigma$ evaluations for PVD-2[W2Ito1] and PVD-2[MT2] comes directly from the number of $\Sigma$ evaluations for the noise integrator method, see Section \ref{sec:new_method}.

\subsection{One dimension}
\label{sec:OneDimensionalProblems}

\begin{figure}[b!]
\includegraphics[width=1.0\linewidth]{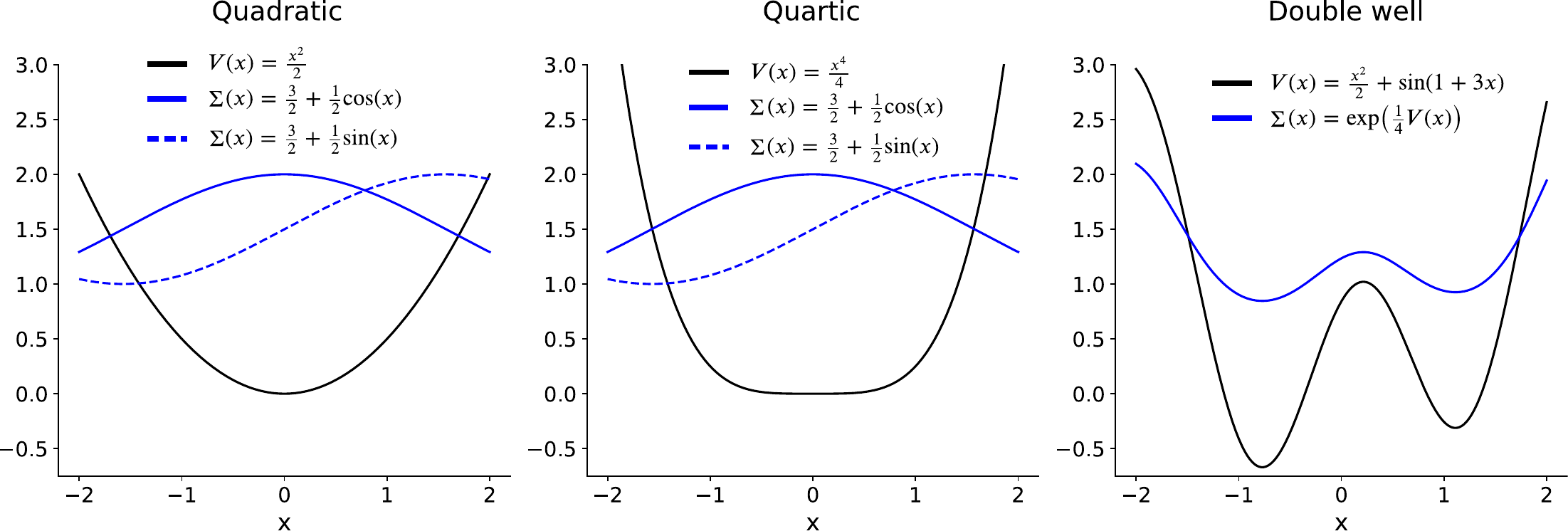}
\caption{The potentials and diffusion coefficients used in one-dimensional experiments. \label{fig:1dV}}
\end{figure}

We consider three potentials of increasing complexity, namely the quadratic potential $V(x) = x^2/2$, the quartic potential $V(x) = x^4/4$ and the asymmetric double-well potential $V(x) = x^2/2 + \sin(1 + 3x)$. For the first two potentials we consider cosine diffusion $\Sigma(x) = \frac{3}{2} + \frac{1}{2}\cos(x)$ and sine diffusion $\Sigma(x) = \frac{3}{2} + \frac{1}{2}\sin(x)$. For the double-well, we consider diffusion of the form $\Sigma(x) = \exp(\frac{1}{4} V(x))$. 
These potentials and diffusion coefficients are illustrated in Figure~\ref{fig:1dV}.
Under mild conditions on $V(x)$, diffusion of the form $\Sigma(x) \propto \exp(2 \sigma^{-2} V(x))$ for $\sigma > 0$ is known to be nearly optimal diffusion for enhancing the crossing rate between metastable wells \cite{Lelievre_Pavliotis_Robin_Santet_Stoltz_2024, Phillips_2024}. Note that $f = - \nabla V$ is globally Lipschitz for the quadratic and double-well potentials, but not for the quartic potential, making it an interesting test case.

For computing the $L_1$ error, we divide the subset $[-5, 5]$ of the $x$-domain into $M=30$ bins and for a fixed $T$ compute the mean error:
\begin{align}
\label{eqn:computingL1Error}
\text{Error}(h, T) \vcentcolon= \frac{1}{M} \sum_{i=1}^M \vert \omega_i - \hat{\omega}_i(h, T) \vert,
\end{align}
where $\omega_i$ is the exact occupancy probability of the $i^{th}$ interval and $\hat{\omega}_i$ is the empirical estimate when running trajectories with fixed stepsize $h$ up to a final time $T$. In all experiments, we set $T = 5 \times 10^7$ and ran each integrator using time steps starting from $10^{-2}$, increasing by a factor of $10^{0.1}$ at each step until the method became unstable. Results for the quadratic, quartic and double-well are shown in Figures \ref{fig:quadraticInvariantMeasure}, \ref{fig:quarticInvariantMeasure} and \ref{fig:doubleWellInvariantMeasure}, respectively. For all curves, we also display an estimate of the Monte-Carlo error in the bias (shaded areas), considering the standard deviation of $10$ independent trajectories. 

Both variants of PVD-2 consistently achieve second-order convergence across various environments, including challenging non-globally Lipschitz cases like the quartic potential. They also yield lower errors than RK4[W2Ito1], using only one force evaluation per step versus RK4[W2Ito1]’s eight. Nevertheless, in one-dimension, time-rescaling combined with the Leimkuhler-Matthews method (LMt) consistently results in the lowest error for any given stepsize. This highlights the importance of transforming a multiplicative noise to additive whenever possible \cite{Phillips_2024}. Note, however, that LMt can only be applied in the multiplicative setting for isotropic diffusion. For this reason, we exclude this method from our higher-dimensional benchmarks.   

\begin{figure}
    \centering
    \includegraphics[width=0.9\linewidth]{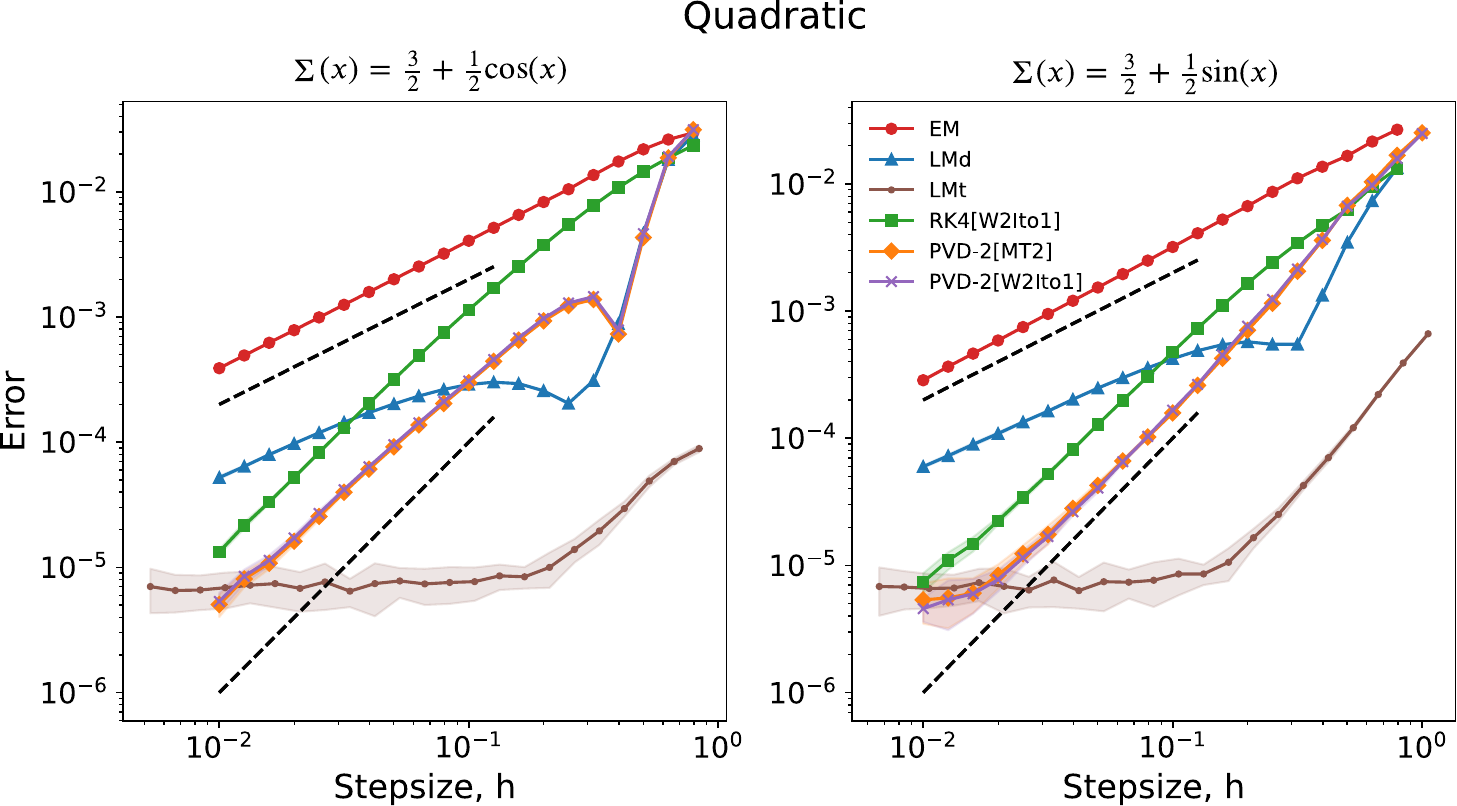}
    \caption{Convergence for sampling the invariant measure in a quadratic potential $V(x) = x^2 / 2$. Left: diffusion $\Sigma(x) = \frac{3}{2} + \frac{1}{2} \cos(x)$. Right: diffusion $\Sigma(x) = \frac{3}{2} + \frac{1}{2} \sin(x)$.}
    \label{fig:quadraticInvariantMeasure}
\end{figure}

\begin{figure}
    \centering
    \includegraphics[width=0.9\linewidth]{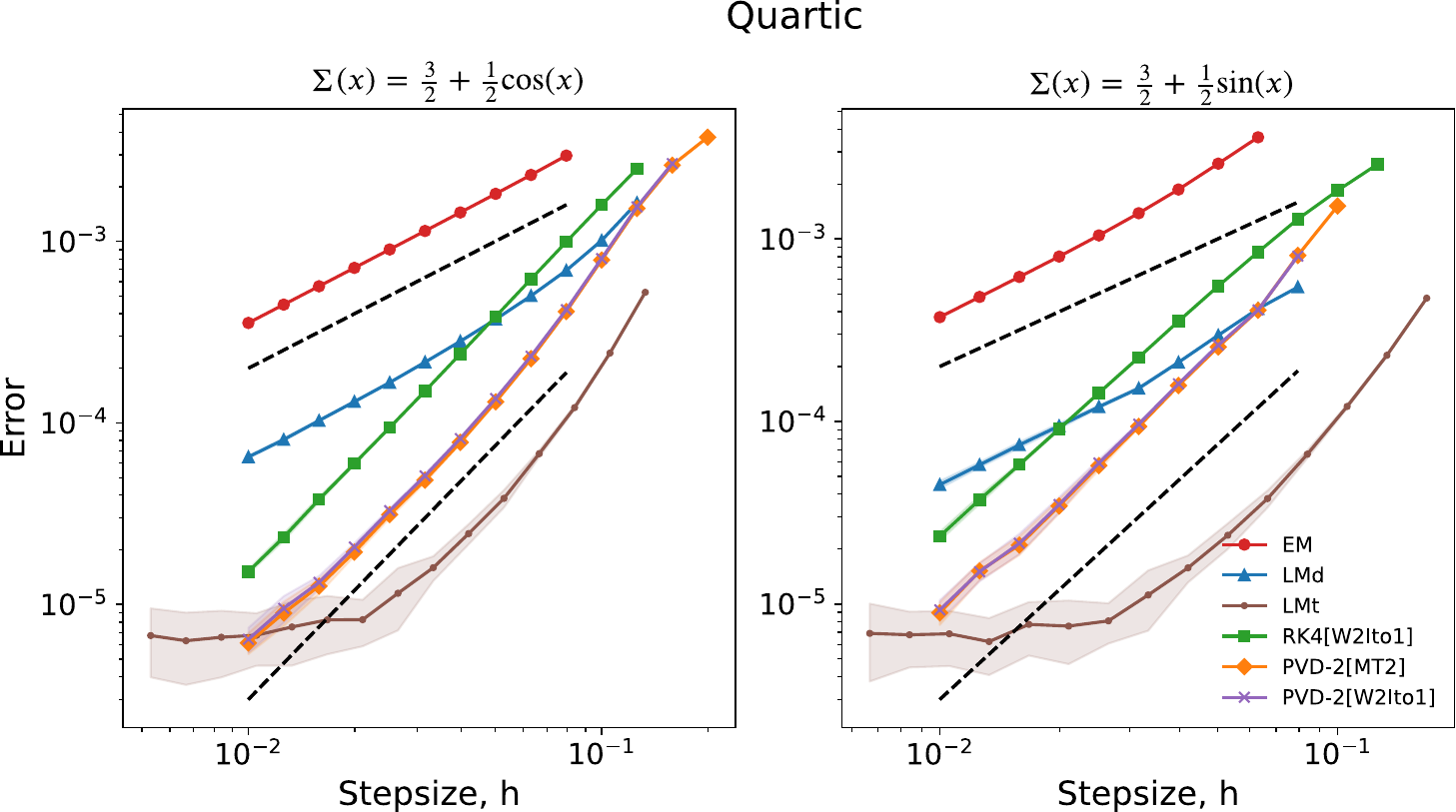}
    \caption{Convergence for sampling the invariant measure in a quartic potential $V(x) = x^4 / 4$. Left: diffusion $\Sigma(x) = \frac{3}{2} + \frac{1}{2} \cos(x)$. Right: diffusion $\Sigma(x) = \frac{3}{2} + \frac{1}{2} \sin(x)$.}
    \label{fig:quarticInvariantMeasure}
\end{figure}

\begin{figure}
    \centering
    \includegraphics[width=0.9\linewidth]{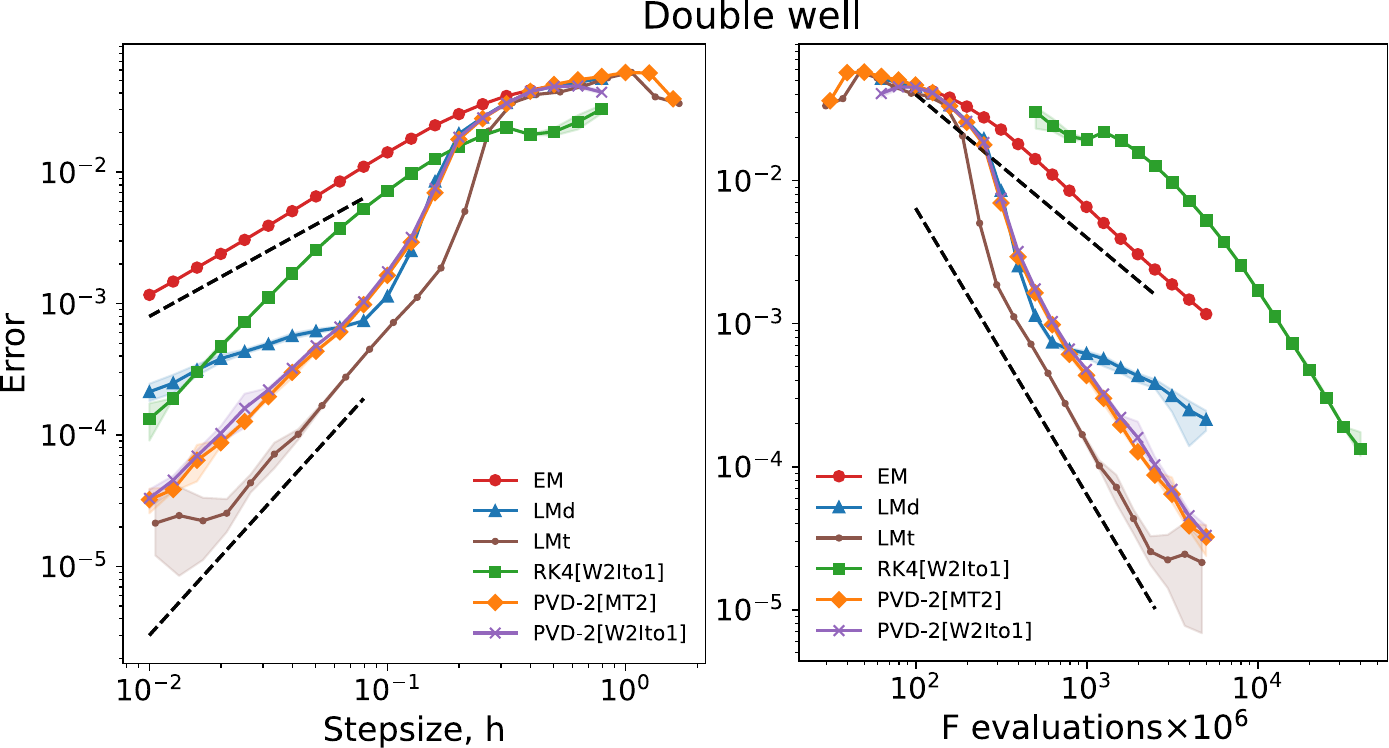}
    \caption{Convergence for sampling the invariant measure in a double-well potential $V(x) = \frac{x^2}{2} + \sin(1+3x)$ with diffusion $\Sigma(x) = \exp{\left(\frac{1}{4}V(x)\right)}$. The left figure shows the error convergence against stepsize and the right figure shows the error convergence against the number of $F$ evaluations, a proxy for computational cost.}
    \label{fig:doubleWellInvariantMeasure}
\end{figure}

\subsection{Two dimensions}

\label{sec:twoDimensions}

We consider a 4-well potential given by
\begin{equation}
\label{eqn:quadrupleWellPotential}
V(x_1,x_2) = \sqrt{\frac{17}{16} -2x_1^2 + x_1^4} + \sqrt{\frac{17}{16} - 2x_2^2 + x_2^4},
\end{equation}
along with four diffusion tensors of increasing complexity:

\begin{enumerate}
    \item[(\textbf{A})] \textit{Constant}: A constant, anisotropic diffusion:
    \[\Sigma_A(x_1,x_2) = \begin{bmatrix}
    2 & 0 \\
    0 & \frac{3}{2}
    \end{bmatrix}.\]

    \item[(\textbf{B})] \textit{Isotropic I}: Non-homogeneous diffusion given by the Moro-Cardin tensor \cite{Moro_1998} which impedes convergence due to the low level of noise in the central high-potential region: 
    
    \[\Sigma_B(x_1,x_2) = \left( 1 + A \exp \left(-\frac{\|x\|^2}{2 \epsilon^2}\right)\right)^{-1} I, \quad
    x=\begin{bmatrix}  x_1 \\ x_2 \end{bmatrix},\]
    where $A=5$ and $\epsilon = 0.3$.

    \item[(\textbf{C})] \textit{Isotropic II}: Non-homogeneous diffusion which aids convergence due to the high level of noise in the central high-potential region:
    \[\Sigma_C(x_1,x_2) = \left( 1 + A \exp \left(-\frac{\|x\|^2}{2 \epsilon^2}\right)\right) I, \quad
    x=\begin{bmatrix}  x_1 \\ x_2 \end{bmatrix},\]
    where $A=1$ and $\epsilon = 0.3$.

    \item[(\textbf{D})] \textit{Anisotropic}: An anisotropic diffusion given by
    \[
    \Sigma_D(x_1,x_2) = I - \frac{x x^T}{2\|x\|^2 + 1},\quad
    x=\begin{bmatrix}  x_1 \\ x_2 \end{bmatrix},
    \]
    which can be written in terms of the planar angle $\theta = \text{arg}(x)$ as
    \[\Sigma_D(x_1,x_2) = I - \frac{\| x \|^2}{2\|x\|^2 + 1}\begin{bmatrix}
\cos^2(\theta) & \cos(\theta)\sin(\theta) \\
\cos(\theta)\sin(\theta) & \sin^2(\theta)
\end{bmatrix}.\]
    The pre-factor of $\frac{\| x \|^2}{2\| x \|^2 + 1}$ ensures that $\theta$-dependent component vanishes at $x = 0$, thus guaranteeing that $\Sigma_D$ is everywhere smooth.
\end{enumerate}

\begin{figure}[t!]
    \centering
    \includegraphics[width=\linewidth]{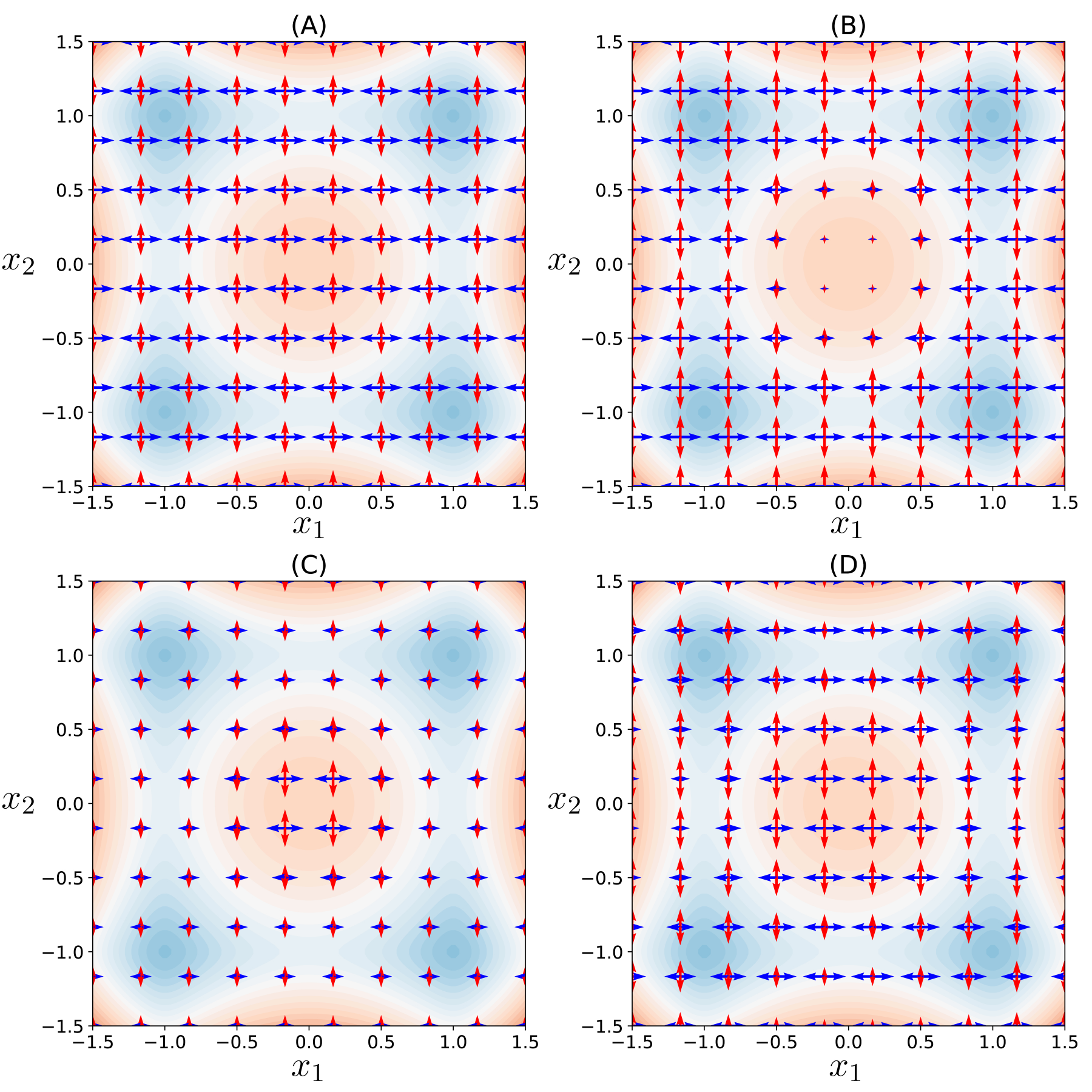}
    \caption{
Contours of the quadruple well potential \eqref{eqn:quadrupleWellPotential} are shown with the various diffusion tensor fields. Diffusion fields $\Sigma(x_1,x_2)$ are visualised by the magnitude of the expected noise increment in the $x_1$
(blue) and $x_2$ (red) directions shown at each grid point. For better comparison, the diffusion arrows in (B) and (D) are scaled by a factor of 2 compared to  (A) and (C).}
    \label{fig:quadrupleWellDiffusions}
\end{figure}

In Figure \ref{fig:quadrupleWellDiffusions}, we visualise how these diffusion tensors vary relative to the energy contours of \eqref{eqn:quadrupleWellPotential}. To measure convergence, we compute the $L_1$ error of the square-norm observable:
\[
\text{Error}(h, N, T) \vcentcolon= \left\vert O - \hat{O}(h,N,T) \right\vert,
\]
where $O = \int (x_1^2 + x_2^2) \rho_\infty(x_1, x_2) dx_1 dx_2$ is the exact square-norm average and $\hat{O}$ is the empirical estimate when running $N$ trajectories with fixed stepsize $h$ and averaging the value of the observable at time $T$. We fix $N=10^5$ and $T = 30$ and ran each integrator with fixed step sizes $h \in \{10^{-2}, 10^{-1.9}, \dots, 10^{-0.1}, 10^{0.0}\}$. Results are shown in Figure \ref{fig:quadrupleWellConvergencePlots}.

Note that in all cases, PVD-2 is the best performing integrator for small stepsizes, within standard error. However, metastability is more severe in the quadruple-well compared to one-dimensional problems (Section \ref{sec:OneDimensionalProblems}). This is especially true for diffusion tensor Isotropic I (B). Here, diffusion vanishes over the central maximum, inhibiting well transitions. This highlights the loss of second-order convergence (for any method) in Figure \ref{fig:quadrupleWellConvergencePlots}(B); the simulation time $T$ is too short to observe complete sampling. In our numerical tests, increasing $T$ did not improve these convergence rates, suggesting that the temporal convergence to equilibrium is very slow for the considered diffusion tensor.

\begin{figure}[t!]
    \centering
    \includegraphics[width=1.0\linewidth]{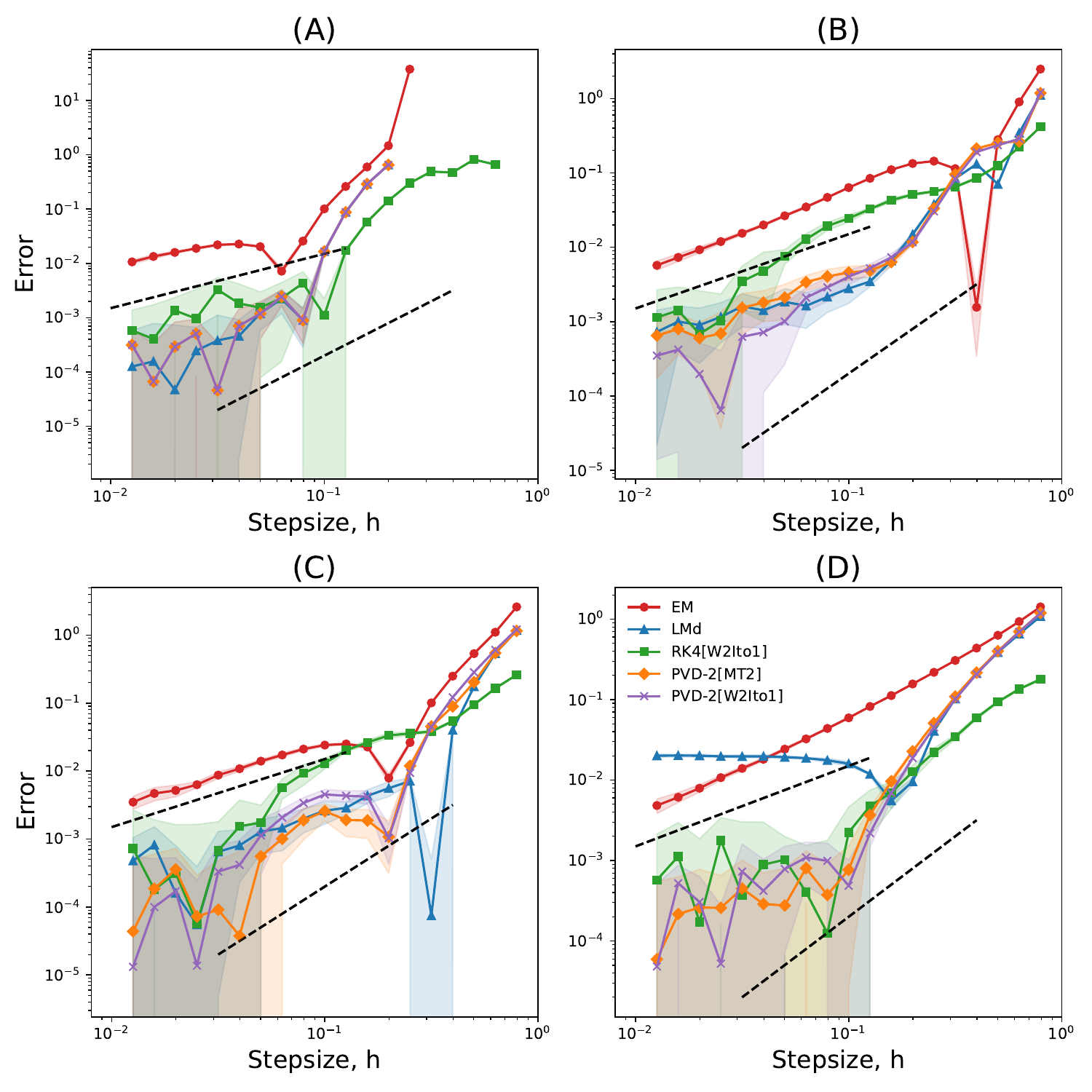}
    \caption{Convergence for sampling the invariant measure in the 2D quadruple-well potential for the four diffusion tensors depicted in Figure \ref{fig:quadrupleWellDiffusions}. Note that in (A), both variants of PVD-2 perform identically hence only PVD-2[W2Ito1] shows.}
    \label{fig:quadrupleWellConvergencePlots}
\end{figure}

\subsection{Higher dimensions}

\begin{figure}[t!]
    \centering
    \includegraphics[width=1.0\linewidth]{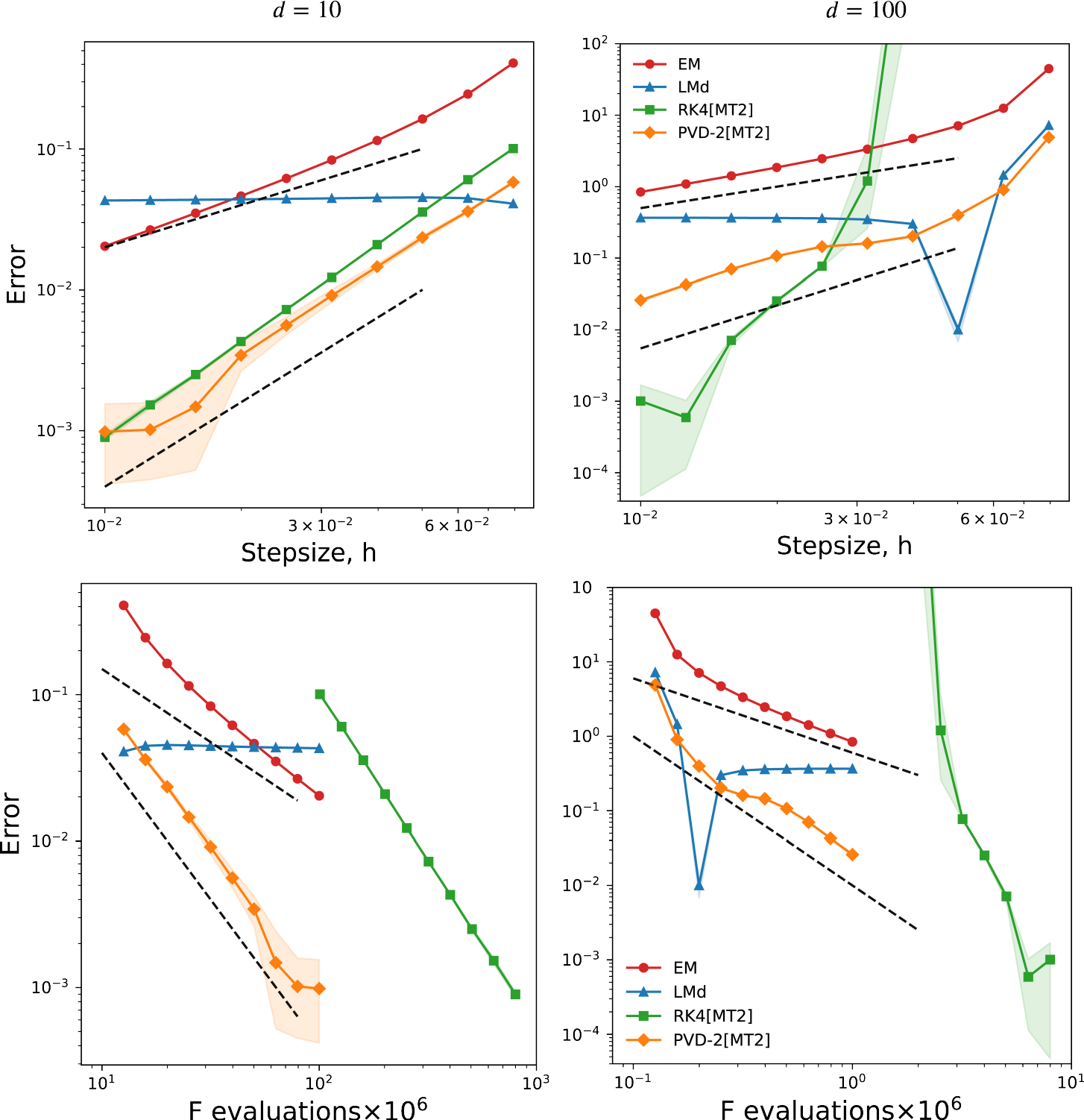}
    \caption{Convergence for sampling the observable $\|x\|$ in the ring-potential \eqref{eqn:ringPotential} with diffusion tensor given by \eqref{eqn:highDimD}. Left: dimension $d=10$. Right: dimension $d=100$.} 
    \label{fig:enter-label}
\end{figure}

We consider the high-dimensional ring potential, which was used as a test problem in \cite[Sect.\thinspace 7]{Sanz-Serna_2014}.
%\cite{Beyn_2014}. 
For $x \in \mathbb{R}^d$, the potential is given by
\begin{align}
\label{eqn:ringPotential}
V(x) = \frac{1}{2}k(1 - \| x \|)^2,
\end{align}
where $\| x \| \vcentcolon = \sqrt{\sum_{i=1}^d x_i^2}$ is the $L_2$ norm of $x$. We use $k=50$ and run experiments for $d=10$ and $d=100$. For the diffusion tensor, we use a rank-one update to the identity matrix which inhibits diffusion in the radial direction, given by 
\begin{align}
\label{eqn:highDimD}
\Sigma(x) = I - \frac{x x^T}{ 2 \| x \|^2}.
\end{align}
This diffusion tensor is non-smooth at $x=0$. However, in high dimensions, and for large $k$, the invariant measure of the ring potential is highly concentrated near the unit sphere and there is thus negligible probability mass near this point. We therefore might still expect to observe second-order convergence when using PVD-2. 

To measure convergence, we compute the $L_1$ error of the square-norm observable:
\[
\text{Error}(h, T) \vcentcolon= \left\vert O - \hat{O}(h,T) \right\vert,
\]
where $O = \int (\sum_{i=1}^d x_i^2) \rho_\infty(x) dx$ is the exact square-norm average and $\hat{O}$ is the empirical estimate when running trajectories with fixed stepsize $h$ up to a final time $T$. For $d=10$ experiments, we set $T=10^6$. For $d=100$, $T=10^4$. We compare the performance of the same integrators as in Section \ref{sec:twoDimensions}, except instead of W2Ito1 we use MT2 with Runge-Kutta 4 Strang splitting (RK4[MT2]) \cite{Abdulle_Vilmart_Zygalakis_2013}. In our low-dimensional experiments, W2Ito1 and MT2 performed very similarly, however MT2 is simpler to implement for large $d$.  

In both high-dimensional experiments, we recovered a second order convergence curve for our new method. In moderate dimensions ($d=10$), we observe that PVD-2[MT2] even outperforms RK4[MT2], whilst requiring only a single $F$ evaluation per step, instead of eight. In high dimensions ($d=100$), although RK4[MT2], it has much more limited stability that all other tested methods. In contrast, LMd does not converge in the multivariate setting.

% \begin{figure}
%     \centering
%     \includegraphics[width=0.5\linewidth]{images/trajectory_in_3D.pdf}
%     \caption{Visualisation of a trajectory of $5000$ steps in $d=3$ with stepsize $h=0.01$ for the ring potential \eqref{eqn:ringPotential} using the method \eqref{eq:new_scheme} with an MT2 noise integrator.}
%     \label{fig:enter-label}
% \end{figure}

\section*{Conclusion}
In this article, we introduced a framework for the construction of a high-order method for sampling the invariant measure of Brownian dynamics with variable diffusion tensor. 
We showed that we can obtain order two with only one force evaluation per time step in spite of the numerous algebraic order conditions.
This work creates new possibilities for accelerating convergence to a target invariant distribution by considering specific diffusion tensors generated in the recent literature.
New interesting questions arise, such as treating various diffusion tensors
important in applications, to (i) address also the reduction of the asymptotic sampling variance, and (ii) take into account the possible stiffness and the application specific structure (e.g. sparsity) of the diffusion tensor.

\section*{Acknowledgements}
The authors would like to thank Charles Matthews for helpful discussions. 
E.B. and G.V. where partially supported by the Swiss National Science Foundation, projects No 200020\_214819 and No. 200020\_192129. 
This work was supported by the United Kingdom Research and Innovation (grant EP/S02431X/1), UKRI Centre for Doctoral Training in Biomedical AI at the University of Edinburgh, School of Informatics. For the purpose of open access, the authors have applied a Creative Commons attribution (CC BY) licence to any author-accepted manuscript version arising. The computations were performed at the University of Edinburgh, and the University of Geneva on the Baobab cluster using the Julia programming language.

%\bibliography{references}

\end{document}